\ifx \processToSlides \undefined
  \ifx\usepackage\RequirePackage        
    \let\usingAmsArtXII\usepackage      
  \fi
\else
  \documentclass[12pt]{amsart}
  \usepackage{amssymb,amscd}

\usepackage[all]{xy}
  \usepackage[letterpaper,margin=2cm,includeheadfoot]{geometry}
  \def \useHugeSize {}
  \def \numberingIsThrough {}
\fi

\ifx \labelsONmargin\undefined

  \ifx\RequirePackage\undefined
    \expandafter\ifx\csname amsart.sty\endcsname\relax
        \documentstyle[12pt,amssymb,amscd]{amsart}
    \fi

    \def\mathbb{\Bbb}
    \def\mathfrak{\frak}
    \def\mathbf{\bold}
    \ifx\boldsymbol\undefined   
      \def\boldsymbol#1{{\bold #1}}
    \fi

    \def\mathbit{\boldsymbol}

    \makeatletter
    \newenvironment{proof}{%
         \@ifnextchar[{%
                       \expandafter\let\expandafter\end@proof
                         \csname endpf*\endcsname
                         \my@proof
                      }{\let\end@proof\endpf\pf}%
        }{\end@proof}
    \def\my@proof[#1]{\@nameuse{pf*}{#1}}
    \makeatother

    \def\xrightarrow[#1]#2{@>{#2}>{#1}>}
    \def\xleftarrow[#1]#2{@<{#2}<{#1}<}

    \def\providecommand#1{\def#1}
    \def\emph#1{{\em #1}}
    \def\textbf#1{{\bf #1}}

    \def\mathring{\overset{\,\,{}_\circ}}
  \else
      \ifx\usepackage\RequirePackage
      \else     
        \documentclass[12pt]{amsart}            
        \usepackage{amssymb,amscd}
        \let\usingAmsArtXII\usepackage

      \fi
      \ifx \mathring\undefined          
        \DeclareMathAccent{\mathring}{\mathalpha}{operators}{"17}
      \fi

      \long\def\FAKEendPROOF{\endtrivlist}
      \ifx\endproof\FAKEendPROOF
          \def\endproof{\qed\endtrivlist}
      \fi

      \ifx\mathbit\undefined    
        \DeclareMathAlphabet{\mathbit}{OML}{cmm}{b}{it}
      \fi

      \ifx \usingAmsArtXII \undefined
      \else                             
        \advance\oddsidemargin -0.1\textwidth
        \evensidemargin\oddsidemargin
      \fi

      \def\Sb#1\endSb{_{\substack{#1}}}
      \def\Sp#1\endSp{^{\substack{#1}}}

  \fi
\makeatletter
\@ifundefined{DeclareFontShape} 
     {\@ifundefined{selectfont}
             {
                \message{We need AmS-LaTeX, this version of LaTeX does not 
                        support it}\@@end
             }{
                \def\mathcal{\cal}
                
                \def\pcyr{%
                        \def\default@family{UWCyr}%
                        \let\oldSl@\sl
                        \def\sl{\def\default@shape{it}\oldSl@}%
                        \cyracc
                        \language\Russian\family{UWCyr}\selectfont
                }
             }}{
                \DeclareFontEncoding{OT2}{}{\noaccents@}
                \DeclareFontSubstitution{OT2}{cmr}{m}{n}
                \DeclareFontFamily{OT2}{cmr}{\hyphenchar\font45 }
                \DeclareFontShape{OT2}{cmr}{m}{n}{%
                     <5><6><7><8><9><10>gen*wncyr %
                     <10.95><12><14.4><17.28><20.74><24.88> wncyr10 %
                }{}
                \DeclareFontShape{OT2}{cmr}{m}{it}{%
                     <5><6><7><8><9><10> gen * wncyi%
                     <10.95><12><14.4><17.28><20.74><24.88> wncyi10%
                }{}
                \DeclareFontShape{OT2}{cmr}{bx}{n}{%
                     <5><6><7><8><9><10> gen * wncyb%
                     <10.95><12><14.4><17.28><20.74><24.88> wncyb10%
                }{}
                \DeclareFontShape{OT2}{cmr}{m}{sl}{%
                     <-> ssub * cmr/m/it%
                }{}
                \DeclareFontShape{OT2}{cmr}{m}{sc}{%
                     <5><6><7><8><9><10>%
                     <10.95><12><14.4><17.28><20.74><24.88> wncysc10%
                }{}
                \DeclareFontFamily{OT2}{cmss}{\hyphenchar\font45 }
                \DeclareFontShape{OT2}{cmss}{m}{n}{%
                     <8><9><10> gen * wncyss%
                     <10.95><12><14.4><17.28><20.74><24.88> wncyss10%
                }{}
                
                \def\cyrencodingdefault{OT2}
                \def\pcyr{%
                        \cyracc
                        \let\encodingdefault\cyrencodingdefault
                        \language\Russian\fontencoding{OT2}\selectfont
                }
             }   

\@ifundefined{theorembodyfont} 
     {
        \def\theorembodyfont#1{\relax}
        \makeatletter
          \let\@@th@plain\th@plain
          \def\th@plain{ \@@th@plain \slshape }
        \makeatother
        \let\normalshape\relax
     }{}   

\ifx\cprime\undefined
     \def\cprime{$'$}
\fi
\makeatletter
  \def\@sect@my#1#2#3#4#5#6[#7]#8{%
\ifnum #2>\c@secnumdepth
   \let\@svsec\@empty
 \else
   \refstepcounter{#1}%
\edef\@svsec{\ifnum#2<\@m
             \@ifundefined{#1name}{}{\csname #1name\endcsname\ }\fi
\noexpand\rom{\csname the#1\endcsname.}\enspace}\fi
 \@tempskipa #5\relax
 \ifdim \@tempskipa>\z@ 
   \begingroup #6\relax
   \@hangfrom{\hskip #3\relax\@svsec}{\interlinepenalty\@M #8\par}%
   \endgroup
   \if@article\else\csname #1mark\endcsname{%
        \ifnum \c@secnumdepth >#2\relax\csname the#1\endcsname. \fi#7}\fi
\ifnum#2>\@m \else
       \let\@tempf\\ \def\\{\protect\\}\addcontentsline{toc}{#1}%
{\ifnum #2>\c@secnumdepth \else
             \protect\numberline{%
               \ifnum#2<\@m
               \@ifundefined{#1name}{}{\csname #1name\endcsname\ }\fi
               \csname the#1\endcsname.}\fi
           #8}\let\\\@tempf
     \fi
 \else
  \def\@svsechd{#6\hskip #3\@svsec
    \@ifnotempty{#8}{\ignorespaces#8\unskip
       \ifnum\spacefactor<1001.\fi}%
        \ifnum#2>\@m \else
          \let\@tempf\\ \def\\{\protect\\}\addcontentsline{toc}{#1}%
            {\ifnum #2>\c@secnumdepth \else
              \protect\numberline{%
                \ifnum#2<\@m
                \@ifundefined{#1name}{}{\csname #1name\endcsname\ }\fi
                \csname the#1\endcsname.}\fi
             #8}\let\\\@tempf\fi}%
 \fi
\@xsect{#5}}
  \ifx\RequirePackage\undefined
  \let\@sect\@sect@my             
  \fi
  \ifx\RequirePackage\undefined 
  \def\th@remark@my{\theorempreskipamount6\p@\@plus6\p@
    \theorempostskipamount\theorempreskipamount
    \def\theorem@headerfont{\it}\normalshape}
    \let\th@remark\th@remark@my
  \else                         
    \let\o@@remark\th@remark

    \ifx\theorempostskipamount\undefined                
    \else                                               
      \def\th@remark{\o@@remark
        \ifdim\theorempostskipamount < 2pt\relax
          \theorempostskipamount\theorempreskipamount
             \multiply\theorempostskipamount\tw@
             \divide\theorempostskipamount\thr@@
        \fi
      }
    \fi
  \fi
\let\myLabel\@gobble
\def\labelsONmargin{\@mparswitchfalse\def\myLabel##1{\@bsphack\marginpar
                                  {\normalshape\tiny\rm Label ##1}\@esphack}}
\ifx \url\undefined
  \def\url#1{{\tt #1}}%
\fi
\def\PREpmodSKIP{\allowbreak  \if@display\mkern18mu\else\mkern8mu\fi}

\def\cyracc{\def\u##1{
                \if \i##1\char"1A%
                \else \if I##1\char"12%
                \else \accent"24 ##1\fi\fi }%
\def\"##1{\if e##1{\char"1B}%
                \else \if E##1{\char"13}%
                \else \accent"7F ##1\fi\fi }%
\def\9##1{\if##1z\char"19 
\else\if##1Z\char"11 
\else\if##1E\char"03 
\else\if##1e\char"0B 
\else\if##1u\char"18 
\else\if##1U\char"10 
\else\if##1A\char"17 
\else\if##1a\char"1F 
\else\if##1p\char"7E 
\else\if##1P\char"5E 
\else\if##1Q\char"5F 
\else\if##1q\char"7F 
\else\if##1i\char"1A 
\else\if##1I\char"12 
\else\if##1N\char"7D 
\fi
\fi
\fi
\fi
\fi
\fi
\fi
\fi
\fi
\fi
\fi
\fi
\fi
\fi
\fi
}%
\def\cydot{{\kern0pt}}}%
\def\cydot{$\cdot$}

\ifx\Russian\undefined
        \def\Russian{0\relax
    \message{Don't know the hyphenation rules for Russian^^J
                        Please do INITeX with `input  russhyph' in the 
                        command line}%
                \gdef\Russian{0\relax}%
        }
\fi
\ifx\RequirePackage\undefined                   
  \def\@putname#1#2#3#4{\def\@@ref{#3}\let\old@bf\bf
        \def\bf##1{\old@bf\if?\noexpand##1?{#4}\else##1\fi}%
        #1{#2}%
        \let\bf\old@bf}
\else                                           
  \def\@putname#1#2#3#4{\def\@@ref{#3}\let\old@bf\bf    
        \let\old@reset@font\reset@font                  
        \def\bf##1{\old@bf\if?\noexpand##1?{#4}\else##1\fi}%
        \def\reset@font##1##2{\old@reset@font##1\if?\noexpand##2?{#4}\else##2\fi}#1{#2}%
        \let\bf\old@bf\let\reset@font\old@reset@font}
\fi
\let\my@ref=\ref
\def\ref#1{\@putname\my@ref{#1}{#1}{\tiny\rm\@@ref}}
\let\my@pageref=\pageref
\def\pageref#1{\@putname\my@pageref{#1}{#1}{\tiny\rm\@@ref}}
\let\my@cite=\cite
\def\cite#1{\@putname\my@cite{#1}{\@citeb}{\tiny\rm\@@ref}}
\makeatother
\fi
\relax 

\ifx \SKIPstatementDEFS \undefined
  \theoremstyle{plain} 
  \theorembodyfont{\sl}
\fi

\ifx\useDoubleSpacing\undefined
\else
  \usepackage{setspace}
\fi

\ifx \address \undefined
  \if \institute \undefined \else       
     \def\address{\institute}
     \ifx \email \undefined
        \let\email\texttt
     \fi
  \fi
\fi

\let\emphOrig\emph
\ifx \doEmphToIndex \undefined
  
\else
  \usepackage{makeidx}
  \makeindex

  \def\eatToBar#1|{}
  \def\emphToIndexSLASH#1\/{\index{#1}\eatToBar}
  \def\emphToIndexDOTSLASH#1.\/{\emphToIndexSLASH #1\/}
  \def\emphAndIndex#1{\emphOrig{#1}{\emphToIndexDOTSLASH #1.\/|}}
  \let\emph\emphAndIndex
\fi

\ifx \SKIPstatementDEFS \undefined

\ifx \numberingIsThrough \undefined
\numberwithin{equation}{section}
\fi

\theorembodyfont{\rm}
\theoremstyle{definition}

\ifx \numberingIsThrough \undefined
\newtheorem{definition}{Definition}[section]
\else
\newtheorem{definition}{Definition}
\fi

\theoremstyle{remark}

\newtheorem{remark}[definition]{Remark} 

\ifx \numberingIsThrough \undefined
\newtheorem{note}{Note}[section] 
\newtheorem{summary}{Summary}[section] 
\else

\fi

\theoremstyle{plain} 
\theorembodyfont{\sl}

\newtheorem{theorem}[definition]{Theorem}
\newtheorem{lemma}[definition]{Lemma}
\newtheorem{corollary}[definition]{Corollary}
\newtheorem{proposition}[definition]{Proposition}

\newcommand{\Hom}{\operatorname{Hom}}
\newcommand{\Ker}{\operatorname{ker}}

\newcommand{\Ext}{\operatorname{Ext}}

\newcommand{\fg}{\mathfrak{g}}
\newcommand{\fq}{\mathfrak{q}}
\newcommand{\fp}{\mathfrak{p}}
\newcommand{\fk}{\mathfrak{k}}
\newcommand{\fs}{\mathfrak{s}}
\newcommand{\fr}{\mathfrak{r}}
\newcommand{\fn}{\mathfrak{n}}
\newcommand{\fb}{\mathfrak{b}}
\newcommand{\fu}{\mathfrak{u}}
\newcommand{\fm}{\mathfrak{m}}
\newcommand{\fl}{\mathfrak{l}}
\newcommand{\fh}{\mathfrak{h}}

\fi 

\author[Ivan Penkov]{\;Ivan Penkov}

\address{
Ivan Penkov
\newline Jacobs University Bremen
\newline Campus Ring 1
\newline 28759 Bremen, Germany}
\email{i.penkov@jacobs-university.de}

\author[Vera Serganova]{\;Vera Serganova}

\address{
Vera Serganova
\newline Department of Mathematics
\newline University of California Berkeley
\newline Berkeley CA 94720, USA}
\email{serganov@math.berkeley.edu}

\usepackage[mathscr]{euscript}
\usepackage{enumerate,color}

\def\clplus{\hbox{$\subset${\raise0.3ex\hbox{\kern -0.55em ${\scriptscriptstyle +}$}}\ }}
\def\crplus{\hbox{$\supset${\raise1.15pt\hbox{\kern -0.55em ${\scriptscriptstyle +}$}}\ }}

\begin{document}
\bibliographystyle{amsplain}

\ifx\useHugeSize\undefined
\else
\Huge
\fi

\relax

\title{Large annihilator category $\mathcal O$ for $\mathfrak{sl}(\infty), \mathfrak{o}(\infty), \mathfrak{sp}(\infty)$}

\date{ \today }

\begin{abstract} We construct a new analogue of the BGG category $\mathcal O$ for the infinite-dimensional Lie algebras
  $\fg=\mathfrak{sl}(\infty),\mathfrak{o}(\infty), \mathfrak{sp}(\infty)$. A main difference with the categories studied in \cite{Nam} and \cite{CP}
  is that all objects of our category satisfy the large  annihilator condition introduced in \cite{DPS}.       Despite the fact  that the splitting Borel subalgebras $\fb$ of $\fg$ are not conjugate, one can eliminate the dependency on  the choice of $\fb$  and introduce a universal  highest weight category  $\mathcal {OLA}$ of $\fg$-modules, the letters $\mathcal{LA}$ coming from "large annihilator".
          The subcategory of integrable objects in  $\mathcal {OLA}$ is precisely the category
  $\mathbb T_{\fg}$ studied in \cite{DPS}. We investigate the structure of $\mathcal {OLA}$, and in particular compute  the
  multiplicities of simple objects in standard objects and the multiplicities of standard objects in indecomposable injectives. We also complete the annihilators in $U(\mathfrak{g})$ of simple objects of $\mathcal{OLA}$.
\end{abstract}

\maketitle

\medskip\noindent {\footnotesize 2010 AMS Subject classification: Primary 17B65, 16S37, 17B55} \\
\noindent {\footnotesize Keywords: BGG category $\mathcal{O}$, finitary Lie algebra, highest weight category, large annihilator condition, standard object, stable Kazhdan-Lusztig multiplicity, Kostka numbers.}

\section{Introduction}

	Let  $\mathfrak{gl}(\infty)$ denote the Lie algebra of finitary infinite matrices over $\mathbb{C}$, and let $\mathfrak{sl}(\infty)\subset\mathfrak{gl}(\infty)$ be the Lie subalgebra of traceless matrices.  One can consider the representation theory of $\mathfrak{sl}(\infty)$ as a way to study stabilization phenomena for representations of the Lie algebras $\mathfrak{sl}(n)$ when $n\to\infty$.  In fact, the very language of representation theory suggests what kind of stabilization features it is natural to consider.  In particular, the theory of tensor $\mathfrak{sl}(\infty)$-modules developed in~\cite{PStyr} shows that Weyl's semisimplicity theorem for $\mathfrak{sl}(n)$ does not stabilize when $n\to\infty$.	 This is because some morphisms of tensor modules over $\mathfrak{sl}(n)$ ``persist at $\infty$'' while others do not.  For instance, the tautological morphism $\mathfrak{sl}(n)\to \mathfrak{gl}(n)$ persists at infinity and induces the tautological injective morphism $\mathfrak{sl}(\infty)\to \mathfrak{gl}(\infty)$. However the morphism of  $\mathfrak{sl}(n)$-modules $\mathbb{C}\to\mathfrak{gl}(n)$ which induces the splitting $\mathfrak{gl}(n)=\mathfrak{sl}(n)\oplus\mathbb{C}$ is lost ``at $\infty$''  as $\mathfrak{gl}(\infty)$ has no nonzero invariants as a module over $\mathfrak{sl}(\infty)$. Similarly, if one considers the Lie algebras $\mathfrak{o}(2n)$ or $\mathfrak{sp}(2n)$, and denotes their natural representations by $V_{2n}$, the respective morphisms ${\bf S}^2(V_{2n})\to \mathbb{C}$ and ${\bf\Lambda}^2(V_{2n})\to \mathbb{C}$ persist at $\infty$, while the (respective) morphisms $\mathbb{C}\to {\bf S}^2(V_{2n})$ and $\mathbb{C}\to {\bf\Lambda}^2(V_{2n})$ are lost at $\infty$.

An intrinsic viewpoint on these phenomena is presented in the  paper [DPS] where a category of tensor modules $\mathbb{T}_{\mathfrak{g}}$ is introduced, and it is established that the tensor products of copies of the natural and conatural representations are injective objects of this category.
	
	Let $\mathfrak{g}=\mathfrak{sl}(\infty), \mathfrak{o}(\infty), \mathfrak{sp}(\infty)$. The purpose of the present paper is to introduce and study an interesting category $\mathcal{OLA}$ of $\mathfrak{g}$-modules which is an analogue of Bernstein-Gelfand-Gelfand's category $\mathcal{O}$ \cite{BGG}, and contains the category of tensor modules $\mathbb{T}_{\mathfrak{g}}$ as a full subcategory. In the papers \cite{Nam} and \cite{CP}, other "analogues at $\infty$" of the category $\mathcal{O}$ have been studied, however these categories are essentially different from the category $\mathcal{OLA}$. In particular, the integrable subcategories of the categories studied in \cite{Nam} and \cite{CP} are semisimple.
	
	  Recall that the category $\mathbb{T}_{\mathfrak{g}}$ consists of integrable $\mathfrak{g}$-modules (i.e., modules which decompose as sums of finite-dimensional modules over any finite-dimensional simple subalgebra of $\mathfrak{g}$) of finite length, satisfying the following three equivalent conditions:	
	\begin{enumerate}[(a)]
		\item $M$ is a weight module for any splitting Cartan subalgebra of $\mathfrak{g}$ (absolute weight module);
		\item $M$ is $\big(\mathrm{Aut}\,\mathfrak{g}\big)^\circ$-invariant, where $\big(\mathrm{Aut}\,\mathfrak{g}\big)^\circ$ is the connected component of the group of automorphisms of $\mathfrak{g}$;
\item the annihilator $\mathrm{Ann}_{\mathfrak{g}}\,m$ of every vector $m\in M$ contains the derived algebra of the centralizer of a finite-dimensional Lie subalgebra of $\mathfrak{g}$.
\end{enumerate}
When one tries to extend $\mathbb{T}_{\mathfrak{g}}$  to an analogue of the BGG category $\mathcal{O}$, one notices that conditions (a) and (b) must be dropped as they no longer hold in the BGG category $\mathcal{O}$.  On the other hand, condition (c) is empty for category $\mathcal{O}$, and therefore, it is the only condition among the three that can lead to an interesting ``category $\mathcal{O}$ for $\mathfrak{g}$''. 
	

More precisely, we fix splitting Cartan and Borel subalgebras $\mathfrak{h}\subset \mathfrak{b}=\mathfrak{h}{\crplus} \mathfrak{n}$ and define the category $\mathcal{OLA}_{\mathfrak{b}}$  by the conditions that its objects are $\mathfrak{h}$-semisimple, satisfy condition c), and are locally finite under the action of any element of $\mathfrak{n}$. The first problem we address, is the dependence  of $\mathcal{OLA}_{\mathfrak{b}}$ on $\mathfrak{b}$. The BGG category $\mathcal{O}$ is independent, up to equivalence, on the choice of a Borel subalgebra as all Borel subalgebras of a finite-dimensional reductive Lie algebra are conjugate. In our case the situation is more complicated and the main result of Section 3 is that there exist Borel subalgebras $\mathfrak{b}$, called perfect, such that for any other splitting Borel subalgebra $\mathfrak{b}'\subset\fg$ the category  $\mathcal{OLA}_{\mathfrak{b}'}$ is naturally equivalent to  $\mathcal{OLA}_{\mathfrak{b}}$ or to a proper full subcategory of  $\mathcal{OLA}_{\mathfrak{b}}$.

In Sections 4-6 we fix a perfect Borel subalgebra $\mathfrak{b}$ of $\fg$ and study  the category  $\mathcal{OLA}=\mathcal{OLA}_{\mathfrak{b}}$. We show that every simple object of  $\mathcal{OLA}$ is a highest weight module and that  $\mathcal{OLA}$ is a highest weight category.  We also prove that every finitely generated object of $\mathcal{OLA}$  has finite length and that any object of  $\mathcal{OLA}$ has an exhaustive socle filtration. Furthermore, we describe the blocks of $\mathcal{OLA}$ and prove that any finitely generated object of $\mathcal{OLA}$ has nonzero annihilator in $U(\mathfrak{g})$. These results manifest further differences with the categories studied in \cite{Nam} and \cite{CP}.

Let us point out that, as a highest weight category, $\mathcal{OLA}$ admits only standard objects and no costandard objects. Costandard objects (analogues of Verma modules) are replaced by certain approximations which do not  "converge" in $\mathcal{OLA}$, nevertheless provide stable Kazhdan-Lusztig multiplicities for a version of BGG-reciprocity which we establish. The indecomposable injectives in  $\mathcal{OLA}$ admit finite filtrations whose successive  quotients are standard objects, while the standard objects have infinite filtrations whose quotients are simple objects. It is essential that the multiplicities of simple objects in standard objects are  finite. 
Interestingly, these latter multiplicities are  a mixture of finite-dimensional Kazhdan-Lusztig numbers and Kostka numbers.

{\bf Acknowledgments.} We are thankful to the referee for pointing out several inaccuracies in the first version of paper, and also for making suggestions for improving its readability. IP has been supported in part by DFG grants PE 980/6-1 and PE 980/7-1. VS has been supported in part by NSF grant 1701532.

        \section{The Set-Up}

The base field is $\mathbb{C}$.  
The notations {\bf S}($\cdot$) and { $\bf\Lambda$}($\cdot$) stand respectively for symmetric and exterior algebra. The superscript $^*$ indicates dual space. Span over a monoid $A$ is denoted by $\langle\cdot \rangle_{A}$.  If $\mu$ is a partition, then $\mathbb S_\mu$ denotes the Schur functor associated with $\mu$. In  particular, $\mathbb{S}_{(k)}(\cdot)= {\bf S}^k(\cdot)$ and $\mathbb{S}_{\underbrace{(1,1,\dots, 1)}_{\text{k times}}}(\cdot)={\bf \Lambda}^k(\cdot)$. The sign \clplus  stands for semidirect sum of Lie algebras (the round part points to the respective ideal).

 We fix a nondegenerate pairing of countable-dimensional vector spaces \texttt{p}$:V\times V_*\to \mathbb{C}$, and define the
Lie algebra $\mathfrak{gl}(\infty)$ as the Lie algebra arising from the associative algebra $V\otimes V_*$.   Both spaces $V$ and $V_*$ carry obvious structures
of $\mathfrak{gl}(\infty)$-modules.  It is a  well known fact (going back to G. Mackey \cite{Mac}) that there exist dual bases $\{v_i\}_{i\in I}$ of $V$ and $\{w_i\}_{i\in I}$ of $V_*$
(i.e. a basis $\{v_i\}_{i\in I}$ of $V$ and a basis $\{w_i\}_{i\in I}$ of $V_*$ such that $\texttt{p}(v_i,w_j)=\delta_{ij}$, where $\delta_{ij}$ is Kronecker's delta) where $I$ is a fixed countable set.  Then clearly $\mathfrak{gl}(\infty)=\langle v_i\otimes w_j|{i,j\in I}\rangle_\mathbb C$.

 By $\mathfrak{sl}(\infty)$ we denote the
Lie algebra $\ker$\texttt{p}; this is a codimension-$1$ Lie subalgebra of $\mathfrak{gl}(\infty)$. Moreover, we fix the abelian subalgebra $$\mathfrak{h}:=\langle h_i:=v_i\otimes w|{i\in I}\rangle_\mathbb C\cap\mathfrak{sl}(\infty)\subset \mathfrak{sl}{(\infty)}.$$

Next, assume that $V$ is endowed with non-degenerate symmetric or antisymmetric form $\texttt{b}:V\otimes V\to\mathbb C$.
If $\texttt{b}$ is symmetric, we define the Lie algebra $\mathfrak {o}(\infty)$ as the vector space ${\bf\Lambda}^2(V)$ with commutator satisfying
$$[u\wedge  v, w \wedge z]=-\texttt{b}(u,w)v\wedge z+ \texttt{b}(u,z)v\wedge w+\texttt{b}(v,w)u\wedge z- \texttt{b}(v,z)u\wedge w.$$
According to \cite{Mac} there exist a basis $\{u, v_i,w_i\}_{i\in I}$ of $V$ such that
\begin{equation}\label{1}
\texttt{b}(u,v_i)=\texttt{b}(u,w_j)=\texttt{b}(v_i,v_j)=\texttt{b}(w_i,w_j)=0,\quad  \texttt{b}(u,u)=1, \quad \texttt{b}(v_i,w_j)=\delta_{ij},
\end{equation}
and a basis $\{v_i,w_i\}_{i\in I}$  of $V$ such that
\begin{equation}\label{2}
\texttt{b}(v_i,v_j)=\texttt{b}(w_i,w_j)=0,\quad \texttt{b}(v_i,w_j)=\delta_{ij}.
\end{equation}
In both cases, we set
$$\mathfrak h:=\langle h_i:=v_i\wedge w_i|{i\in I}\rangle_\mathbb C\subset \mathfrak{o}(\infty).$$

If $\texttt{b}$ is antisymmetric, we define the Lie algebra $\mathfrak{sp}(\infty)$ as the space ${\bf S}^2(V)$ with commutator satisfying
$$[u  v, w  z]=\texttt{b}(u,w)v z+ \texttt{b}(u,z)v w+\texttt{b}(v,w)u z+ \texttt{b}(v,z)u w.$$
Furthermore, there exists a basis  $\{v_i,w_i\}_{i\in I}$
of $V$ satisfying (\ref{2}).
We set
$$\mathfrak h:=\langle h_i:=v_i w_i|{i\in I}\rangle_\mathbb C\subset\mathfrak{sp}(\infty).$$

We denote by $\mathfrak g$ one of the Lie algebras $\mathfrak{sl}(\infty)$,  $\mathfrak {o}(\infty)$ or $\mathfrak{sp}(\infty)$. In all four cases above, $\mathfrak{h}$ is a \emph{splitting Cartan subalgebra of} $\mathfrak{g}$ according to \cite{DPS}.
Furthermore, $\mathfrak g$ has a root decomposition
$$\mathfrak g=\mathfrak h\oplus\bigoplus_{\alpha\in\Delta}\mathfrak g_{\alpha},$$
where $\Delta$ is the \textit{root system of} $\mathfrak g $.
We define $\varepsilon_i\in\mathfrak h^*$ by setting
$$\varepsilon_i(h_j):=\delta_{ij}.$$
Then the root system  of $\mathfrak{sl}(\infty)$ is  $$\Delta=A_{\infty}=\{\varepsilon_i-\varepsilon_j\,|\,i,j\in I, i\neq j\},$$
and the root system of $\mathfrak{sp}(\infty)$ is
$$\Delta=C_\infty =\{\varepsilon_i-\varepsilon_j\,|\,i,j\in I, i\neq j\}\cup\{\pm(\varepsilon_i+\varepsilon_j)\,|\,i,j\in I\}.$$
The Lie algebra $\mathfrak{o}(\infty)$ has two root systems depending on whether $\mathfrak h$ is of type $B$ or type $D$:
$$\Delta=B_\infty=\{\varepsilon_i-\varepsilon_j\,|\,i,j\in I, i\neq j\}\cup\{\pm(\varepsilon_i+\varepsilon_j)\,|\,i,j\in I, i\neq j\}\cup\{\pm\varepsilon_i\,|\,i\in I\}$$
if (\ref{1}) holds, and
$$\Delta=D_\infty=\{\varepsilon_i-\varepsilon_j\,|\,i,j\in I, i\neq j\}\cup\{\pm(\varepsilon_i+\varepsilon_j)\,|\,i,j\in I, i\neq j\}$$
if (\ref{2}) holds.

For a $\fg$-module $M$ which is semisimple as an $\fh$-module, we put $$\operatorname{supp}M:=\{\lambda\in \fh^*\,| M_{\lambda}:=\{m\in M| hm=\lambda(h)m \; \forall h\in \fh\}\neq 0\}.$$

Next, set

\[\tilde I := \left\{ \begin{array}{cc}
	\{\varepsilon_i\,|\,i\in I\}, &\textnormal{ for }\mathfrak g=\mathfrak{sl}(\infty) \\
	\{\pm\varepsilon_i\,|\,i\in I\}, &\textnormal{ for }
\mathfrak g=\mathfrak{o}(\infty) \textnormal {, } \mathfrak{sp}(\infty).
	\end{array}\right. \]
Note that $V$, as well as $V_*$ for $\fg=\mathfrak{sl}(\infty)$ is a $\fg$-module which is semisimple as an $\fh$-module. We refer to $V$ (respectively, $V_*$) as the {\it natural} (respectively, {\it conatural}) {\it $\fg$-module}. In all cases except $\Delta=B_{\infty}$, we have  $\operatorname{supp}V=\tilde I$. If $\Delta=B_{\infty}$ then $\operatorname{supp}V=\tilde I \sqcup 0$. Finally, $\operatorname{supp}V_*=-\tilde I$ for $\fg=\mathfrak{sl}(\infty)$ (note that the pairing \texttt{p} makes $V_*$ a $\fg$-submodule of $V^* = \operatorname{Hom}_{\mathbb{C}}(V,\mathbb{C}$)).

For $\mathfrak g=\mathfrak{o}(\infty), \mathfrak{sp}(\infty)$ we call a subset $J$ of $\tilde I$ \emph{symmetric} if $J=-J$.
For any  subset $J\subset \tilde I$, which we assume symmetric if $\mathfrak{g}=\mathfrak{o}(\infty)$, $\mathfrak{sp}(\infty)$, put 
$$\Delta_J:=\Delta\cap\langle J\rangle_\mathbb Z,$$
and let $\mathfrak{g}_J$ be the root subalgebra of $\fg$ generated by $\fg_{\alpha}$ for  $\alpha\in\Delta_J$. By $\fg^c_J$ we denote the centralizer of
$\fg_{\tilde I\setminus J}$ in $\fg$. For the root systems $C_\infty$ and $D_\infty$  we have $\fg^c_J=\fg_{J}$. This holds also for $A_\infty$  under the assumption that $J$ is not cofinite in $\tilde I$, otherwise $\fg_{J}=[\fg^c_J,\fg^c_J]$.
For the root system $B_\infty$, we have $\fg^c_J\subset\fg_J$: if $\fg_J$ has root system $B_{|J|/2}$, then $\fg^c_J$ has root system $D_{|J|/2}$ where $|J|=\operatorname{card}J$ (if $|J|<\infty$, the
root systems $B_{|J|/2}$ and $D_{|J|/2}$ are the classical finite root systems of respective types $B$ or $D$).

A \emph{splitting Borel subalgebra} $\mathfrak b$ containing $\mathfrak h$ \cite{DP}, has the form
$$\mathfrak b=\mathfrak h\oplus\bigoplus_{\alpha\in\Delta^+}\mathfrak g_{\alpha}$$
for an arbitrary decomposition $\Delta=\Delta^+ \sqcup \Delta^-$ such that $\Delta^-=-\Delta^+$ and $\alpha+\beta\in \Delta^+$ whenever $\alpha, \beta\in \Delta^+$, $\alpha+\beta\in\Delta$.

All splitting Borel subalgebras containing $\mathfrak h$ are in a natural bijection with the set of total orders $\prec$ on $\tilde I$, 
subject to the condition that $a\prec b$ implies $-b\prec -a$ in the case $\mathfrak g=\mathfrak{o}(\infty)$ or $\mathfrak{sp}(\infty)$. In what follows, we call such
orders \emph{symmetric} or $\mathbb Z_2$-\emph{linear}. Indeed, given a (symmetric) total order $\prec$ on $\tilde I$, we set 

\[\Delta^+ := \left\{ \begin{array}{cc}
	\{\varepsilon_i-\varepsilon_j\,|\, \varepsilon_i<\varepsilon_j\} &\textnormal{ if } \Delta=A_{\infty}, \\
	\{\alpha\,|\, \alpha\prec -\alpha\}\sqcup\{\alpha+\beta\,|\,\alpha\prec -\alpha,\beta\prec -\beta\}\sqcup&\\
\sqcup\{\alpha-\beta\,|\,\alpha\prec\beta\}\textnormal{ for } \alpha, \beta\in \tilde I & \textnormal{ if }
\Delta=B_\infty,\\
\{2\alpha\,|\,\alpha\prec -\alpha\}\sqcup\{\alpha+\beta\,|\,\alpha\prec -\alpha,\beta\prec -\beta\}\sqcup&\\
\sqcup\{\alpha-\beta\,|\,\alpha\prec\beta\} \textnormal{ for } \alpha, \beta\in \tilde I &\textnormal{ if }
\Delta=C_\infty,\\
\{\alpha+\beta\,|\,\alpha\prec -\alpha,\beta\prec -\beta\}\sqcup&\\
\sqcup\{\alpha-\beta\,|\,\alpha\prec\beta\}\textnormal{ for } \alpha, \beta\in \tilde I &\textnormal{ if }
\Delta=D_\infty.
	\end{array}\right. \]
In the remainder of the paper we assume that all total orders $\prec$ on $\tilde I$  considered are symmetric for $\fg=\mathfrak{o}(\infty)$, $\mathfrak{sp}(\infty)$.

Given a total order $\prec$ on the set $\tilde I$, we define subsets $S_{max}$ and $S_{min}$ of $\tilde I$ as follows: $S_{min}$ (respectively, $S_{max}$) is the set of all $\alpha\in \tilde I$ such
that there exists a cofinite subset $A\subset \tilde I$ in which  $\alpha$ is minimal (respectively, maximal). 
Note that for $\mathfrak g=\mathfrak{o}(\infty)$, $\mathfrak{sp}(\infty)$, we have $S_{min}=-S_{max}$.
A total order $\prec$ on $\tilde I$ is  {\it ideal} if both $S_{min}$ and $S_{max}$ are infinite; a total order $\prec$ on $\tilde I$ is  {\it perfect} if it is ideal and $\tilde I=S_{min}\cup S_{max}$.  The corresponding Borel subalgebras are also called  \textit{ideal} or \textit{perfect}. Note that all perfect total orders on $\tilde I$ are isomorphic, which implies that all perfect Borel subalgebras are conjugate under $\textup{Aut}\mathfrak{g}$. 

A root $\alpha \in \Delta^+$ is \emph{simple} if $\alpha$ cannot be decomposed as a sum $\beta+\gamma$ for $\beta, \gamma \in \Delta^+$. If  a root  can be written as a linear combination of simple roots we call it a $\mathfrak{b}$-\emph{finite} root. All other roots are {\it infinite} by definition. For instance, if $\fb$ is perfect with positive roots $\varepsilon_i - \varepsilon_j$ for $\varepsilon_i\prec\varepsilon_j$ for $\fg=\mathfrak{sl}{(\infty)}$, then the $\fb$-finite roots are of the form $\varepsilon_i - \varepsilon_j$ for $\varepsilon_i, \varepsilon_j\in S_{min}$ or $\varepsilon_i, \varepsilon_j\in S_{max}$.

If $M$ is a $\fg$-module for $\fg=\mathfrak{sl}{(\infty)}, \mathfrak{o}{(\infty)}, \mathfrak{sp}{(\infty)} $, or for a finite-dimensional Lie algebra $\fg$, the Fernando-Kac subalgebra $\fg[M]$ of $\fg$ consists of all vectors $g\in\fg$ which act locally finitely on $M$, i.e. such that $\textup{dim}(\langle m,g{m},g^2{m},\dots\rangle_\mathbb C)<\infty$ for any $m\in M$. The fact that $\fg[M]$ is indeed a Lie subalgebra has been proved independently in \cite{K} and \cite{Fe}.

We say that a $\fg$-module $M$ satisfies the \emph{large annihilator condition} if, for any $m\in M$, the annihilator in $\fg$ of $m$ contains the  commutator subalgebra of the centralizer of a finite-dimensional Lie subalgebra of $\fg$ (i.e. if $M$ satisfies condition $($c$)$ from the Introduction).

Finally, recall that the \emph{socle}  of a $\fg$-module $M$, soc$M$, is the sum of all simple submodules of M. It is a standard fact that soc$M$ is the largest semisimple submodule of $M$. The \emph{socle fltration} of $M$ is
$$ 0\subset \textup{soc}M=\textup{soc}^0M\subset \textup{soc}^1M\subset\textup{soc}^2M \subset \dots $$
where $ \textup{soc}^iM:=\pi^{-1}_i( \textup{soc}(M/ \textup{soc}^{i-1}M))$ and $\pi_i:M\to M/ \textup{soc}^{i-1}M$ is the canonical homomorphism. We say that the socle filtration of $M$ is \emph{exhaustive} if $M=\bigcup_{i\geq0} \textup{soc}^iM$.

\section{The category $\mathcal{OLA_{\mathfrak b}}$}

Let $\fh$ be the fixed splitting Cartan subalgebra of $\fg$, see Section 2, and $\fb=\fh\oplus\fn$ be a fixed splitting Borel subalgebra containing $\fh$ and corresponding to a total order $\prec$ on $\tilde I$. We define $\mathcal{OLA_{\mathfrak b}}$ as the full subcategory of the category of all $\mathfrak{g}$-modules, consisting of $\fg$-modules $M$ satisfying the following conditions:
\begin{enumerate}[(i)]
\item $M$ satisfies the large annihilator condition;
\item $M$ is $\mathfrak{h}$-semisimple;
\item every $x\in\mathfrak n$ acts locally nilpotently on $M$.
\end{enumerate}

The first problem we address, is to what extent $\mathcal{OLA}_{\mathfrak b}$ depends on the choice of $\mathfrak{b}$.

Set $S:=S_{min}\cup S_{max}\subset\tilde I$.  For $a$, $b\in \mathbb{Z}_{\geq0}$, define $S_{min}(a)\subset S_{min}$ and $S_{max}(b)\subset S_{max}$ to be respectively the first $a$ elements of $S_{min}$ and the last $b$ elements of $S_{max}$.  Here we assume $S_{min}(0)=S_{max}(0)=\emptyset$.
 Put $\fg_{a,b}:=\fg_{\tilde I\setminus (S_{min}(a)\cup S_{max}(b))}$,
 where for $\mathfrak g=\mathfrak{o}(\infty)$ or $\mathfrak{sp}(\infty)$ we suppose that $a=b$ and  that all subsets of $\tilde I$ we consider are symmetric.

The large annihilator condition can be rewritten in the form
\begin{equation}\label{anncond}
 \text{ for every}\,\,m\in M\,\ \text{there exists a cofinite set}\,\, J\subset \tilde I \,\text{such that}\,\,\fg^c_{J}m=0. 
  \end{equation}

\begin{lemma}\label{fk} Let $M\in\mathcal{OLA}_{\mathfrak b}$.
\begin{enumerate}[(a)]
\item If $M$ is finitely generated, then there exist $a,b\in\mathbb Z_{\geq0}$  such that  $\fg_{a,b}\subset \fg{[M]}$.

\item  For an arbitrary $M$, we have $\mathfrak g_{\tilde I\setminus S} \subset \fg{[M]}$. 
\end{enumerate}
\end{lemma}
\begin{proof} 
  It suffices to prove (a) for a cyclic module. Let $M$ be generated by a vector $m\in M$. By (\ref{anncond}) there exists a cofinite set $J\subset \tilde I$
  such that $\fg^c_{J}m=0$.
  Since the action of $\operatorname{ad} x$ on $U(\fg)$ is locally finite for all  $x\in\fg$, and $M=U(\fg)m$, we conclude that $\fg_{J}^c\subset\fg[M]$.
  On the other hand, by (iii) we have $\fn\subset\fg[M]$. It is easy to check that the subalgebra of $\fg$ generated by $\fn$ and $\fg_{J}^c$ equals 
  $\fg_{J'}$ where $J'$ is the minimal interval containing $J$. By the cofiniteness of $J'$ we get  $\fg_{J'}=\fg_{a,b}$ for some  $a,b\in\mathbb Z_{\geq0}$.
  Hence $\fg_{a,b}\subset \fg{[M]}$.

(b) is a consequence of (a) since $\fg_{\tilde I\setminus S}$ equals the intersection $\bigcap_{a,b}\fg_{a,b}$.
\end{proof}

\begin{theorem}\label{qr} Assume that   $\mathfrak b$ is ideal, and let $J$ and $K$ be infinite (symmetric) subsets of $\tilde I$ such that
  $\tilde I=J\sqcup K$.  Suppose further that $S\subset K$, and set   $\mathfrak b_K:=\mathfrak g_K\cap \mathfrak b$. Let $\mathcal{OLA}_{\mathfrak b_K}$ be the 
category of $\mathfrak g_K$-modules satisfying the conditions (i)--(iii) with respect to $\mathfrak{b}_K$. 

 (a)The categories $\mathcal{OLA}_{\mathfrak b_K}$ and $\mathcal{OLA}_{\mathfrak b}$ are equivalent.

(b) If the root system of $\fg$ is $B_{\infty}$, there is also an equivalence of the categories $\mathcal{OLA}_{\mathfrak b_K^c}$ and $\mathcal{OLA}_{\mathfrak b}$
where $\fb^c_K:=\fb\cap\fg^c_K$.
\end{theorem}
\begin{proof} (a) Consider the functor $$\Phi_K:\mathcal{OLA}_{\mathfrak b}\to\mathcal{OLA}_{\mathfrak b_K},\, \Phi_K(M):=M^{\fg^c_J},$$ where the superscript $(\cdot)^{\fg^c_J}$ indicates taking invariants.
We shall prove that $\Phi_K$ is an equivalence of categories.

Let $\mathcal{OLA}^{a,b}_{\fb}$ denote the full subcategory of $\mathcal{OLA}_{\fb}$ consisting of modules such that $\fg_{a,b}\subset\fg[M]$.
By Lemma \ref{fk}(a), $$\mathcal{OLA}_{\fb}=\lim_{\longrightarrow} \mathcal{OLA}^{a,b}_{\fb}.$$ Similarly, we define the category $\mathcal{OLA}^{a,b}_{\fb_K}$
as the subcategory of modules $M$ satisfying $\fg_{a,b,K}:=\fg_K\cap\fg_{a,b}\subset\fg[M]$ . Then
$$\mathcal{OLA}_{\fb_K}=\lim_{\longrightarrow} \mathcal{OLA}^{a,b}_{\fb_K}.$$

Clearly, $\Phi_K$ induces well-defined functors $$\Phi_K^{a,b}:\mathcal{OLA}^{a,b}_{\fb}\to \mathcal{OLA}^{a,b}_{\fb_K},$$ and it suffices to prove that $\Phi_K^{a,b}$ are equivalences of categories for all
$a,b\in\mathbb Z_{\geq 0}$. Denote by $\tilde{\mathbb{T}}_{\fg_{a,b}}$ the inductive completion of the category $\mathbb{T}_{\fg_{a,b}}$.
Then, for any fixed $a$, $b\in \mathbb{Z}_{\geq 0}$, we have the following commutative diagram of functors
$$\begin{CD}\mathcal{OLA}^{a,b}_{\fb}@>\Phi_K^{a,b}>>\mathcal{OLA}^{a,b}_{\fb_K}  \\@V\operatorname{Res}_{\fg_{a,b}}VV @VV\operatorname{Res}_{\fg_{a,b,K}}V\\ 
\tilde{\mathbb T}_{\fg_{a,b}}@>\Phi_K^{a,b}>>\tilde{\mathbb T}_{\fg_{a,b,K}}.\end{CD}$$

We claim that $\Phi_K^{a,b}$ is an equivalence of symmetric monoidal categories downstairs. This follows directly from Lemma 5.13 and 5.14 in  \cite{PS} which  prove
that $\Phi_K^{a,b}$ establishes an equivalence between ${\mathbb T}_{\fg_{a,b}}$ and ${\mathbb T}_{\fg_{a,b,K}}$. The passage to the respective inductive completions   
$\tilde{\mathbb T}_{\fg_{a,b}}$ and $\tilde{\mathbb T}_{\fg_{a,b,K}}$ is automatic because  $\Phi_K^{a,b}$ commutes with direct limits.

Next we show that $\Phi_K^{a,b}$ remains an equivalence upstairs. Indeed,
consider the decomposition (of vector spaces) $\fg=\fg_{a,b}\oplus\fr$, where $\fr$ is a $\fg_{a,b}$-stable subspace. The objects of $\mathcal{OLA}^{a,b}_{\fb}$ are pairs
$(M,\varphi)$ where $M\in\tilde{\mathbb T}_{\fg_{a,b}}$ and
$\varphi:M\otimes\fr\to M$ is a morphism satisfying a certain set of tensor identities. Note that $\fg_K=\Phi_K(\fg)$, and set $\fr_K:=\Phi_K(\fr)$. We have
$\fg_K=\fg_{a,b,K}\oplus\fr_K$.
The objects of $\mathcal{OLA}^{a,b}_{\fb_K}$ are pairs $(N,\psi)$ where 
$N\in\tilde{\mathbb T}_{\fg_{a,b,K}}$ and $\psi:N\otimes\fr_K\to N$ is a morphism satisfying the same set of tensor identities. Obviously,
$\Phi_K^{a,b}(\fr)=\fr_K$ and $\Phi_K^{a,b}(\varphi)=\psi$. This completes the proof of (a).

To prove (b), define the functor $$\Phi'_K:\mathcal{OLA}_{\mathfrak b}\to\mathcal{OLA}_{\mathfrak b^c_K},\, \Phi'_K(M)=M^{\fg_J},$$  The proof that $\Phi'_K$ is an
equivalence of categories is similar to the proof of (a).
\end{proof}
\begin{corollary}\label{cor:perfect} If $\fb\subset\fg$ is an ideal subalgebra, the category $\mathcal{OLA}_{\mathfrak b}$ is equivalent to the category 
$\mathcal{OLA}_{\mathfrak b'}$ for a perfect subalgebra $\fb'\subset\fg$.
\end{corollary}
\begin{proof} First we prove that $\mathcal{OLA}_{\fb}$ is equivalent to $\mathcal{OLA}_{\fb_S}$.
  If $S$ is coinfinite in $\tilde I$, this is established in Theorem \ref{qr}(a).  Therefore, assume that $S$ is cofinite in $\tilde I$.
  Extend $\tilde I$ to a totally ordered set  $\tilde P$ by replacing the interval $\tilde I\setminus S$
  by an infinite interval (symmetric in the case $\fg=\mathfrak{o}(\infty)$ or $\mathfrak{sp}(\infty)$). Then $\fg$ and $\fg_{S}$ are embedded into an
  isomorphic copy $\fg_{\tilde P}$ of $\fg$ in which the role of $\tilde I$ is played by $\tilde P$.
  Let $\tilde \fb$ be the Borel subalgebra of  $\fg_{\tilde P}$ defined by the ordered set $\tilde P$. Now  Theorem \ref{qr} implies
  that both  categories $\mathcal{OLA}_{\fb_S}$ and  $\mathcal{OLA}_{\fb}$ are equivalent to $\mathcal{OLA}_{\tilde\fb}$. Hence,
  $\mathcal{OLA}_{\fb_S}$ and  $\mathcal{OLA}_{\fb}$ are equivalent.
  
  Furthermore, $\fb_S$ is a perfect Borel subalgebra of $\fg_s$ and $\fg_s\simeq\fg$. Consider an isomorphism $\varphi:\fg_S\to \fg$ and set $\fb':=\varphi(\fb_s)$.
  This isomorphism extends to an equivalence between  $\mathcal{OLA}_{\fb_S}$ and  $\mathcal{OLA}_{\fb'}$. The statement follows.
  \end{proof}
  \begin{corollary}\label{cor:BD} Assume that the root system of $\fg$ is $B_{\infty}$ and the root system of $\fg'$ is $D_\infty$. Then, for any ideal Borel
    subalgebra $\fb\subset\fg$ there exists a  perfect subalgebra $\fb'\subset\fg'$ such that the category $\mathcal{OLA}_{\mathfrak b}$ is equivalent to the category 
$\mathcal{OLA}_{\mathfrak b'}$. 
\end{corollary}
\begin{proof} The proof is similar to the proof of Corollary \ref{cor:perfect} via application of Theorem \ref{qr}(b).
\end{proof}

In the rest of the section, $\fb$ is an arbitrary splitting Borel subalgebra containing $\fh$.

\begin{proposition}\label{prop:fin} Assume that $S$ is finite. Then there exists a perfect Borel subalgebra
  $\fb'\subset\fg$ such that $\mathcal{OLA}_{\fb}=\mathcal{OLA}^{a,b}_{\fb'}$ for some $a,b\geq 0$. In particular, if $S=\emptyset$ then
  $\mathcal{OLA}_{\fb}=\tilde{\mathbb T}_{\fg}$.
  \end{proposition}
  \begin{proof} Set $a=|S_{min}|$, $b=|S_{max}|$. Define a perfect order on $\tilde I$ such that $S_{min}\subset \tilde I$
    (respectively, $S_{max}\subset\tilde I$) are the first  (respectively, the last) elements of $\tilde I$. Denote by $\fb'$ the Borel subalgebra corresponding
    to this order. Then 
    $\mathcal{OLA}^{a,b}_{\fb}=\mathcal{OLA}^{a,b}_{\fb'}$, and by Lemma \ref{fk}(b) $\mathcal{OLA}_{\fb}=\mathcal{OLA}^{a,b}_{\fb}$.
 The assertion follows.
  \end{proof}

  \begin{proposition}\label{prop:onesided} Let $\fg=\mathfrak{sl}(\infty)$. Suppose that exactly one of $S_{min}$ and $S_{max}$ is finite. 

(a) The categories $\mathcal{OLA}_{\mathfrak b_S}$ and $\mathcal{OLA}_{\mathfrak b}$ are equivalent.

(b) Set 
$$\mathcal{OLA}^{a,\infty}_{\fb}:=\displaystyle\lim_{\longrightarrow}\mathcal{OLA}^{a,b}_{\fb}\;\text{for}\;b\rightarrow\infty,$$ $$\mathcal{OLA}^{\infty,b}_{\fb}:=\displaystyle\lim_{\longrightarrow}\mathcal{OLA}^{a,b}_{\fb}\;\text{for}\;a\rightarrow\infty.$$
Then there exists a perfect Borel subalgebra $\mathfrak b'\subset \fg$ such that
\begin{enumerate}
\item if $|S_{min}|=a$ and $S_{max}$ is infinite, then $\mathcal{OLA}_{\fb}$ is equivalent to $\mathcal{OLA}^{a,\infty}_{\fb'}$;
\item if $S_{min}$ is infinite and $|S_{max}|=b$, then $\mathcal{OLA}_{\fb}$ is equivalent to 
$\mathcal{OLA}^{\infty,b}_{\fb'}$.
\end{enumerate}
\end{proposition}
\begin{proof} (a) can be proven in the same way as Corollary \ref{cor:perfect}, and we leave the proof to the reader.

  Let us prove (b) in the case (1). Case (2) is similar. By (a) we may assume that $\tilde I=S$.
  We include $S_{min}$ into an ordered set $L$ isomorphic to $\mathbb Z_{\geq 0}$ such that $S_{min}$ is identified with the first $a$ elements of $L$. Set
  $\tilde P:=L\sqcup S_{max}$, $L\prec S_{max}$ and consider the corresponding Lie algebra $\fg_{\tilde P}$ with Borel subalgebra $\tilde \fb$.  
  Define the functor $$\Phi_S:\mathcal{OLA}_{\tilde \fb}\to\mathcal{OLA}_{\mathfrak b},\quad \Phi_S(M):=M^{\fg_{L\setminus S_{min}}}.$$ As in the proof of Theorem \ref{qr}, one can show
  that the restriction of $\Phi_S$ to $\mathcal{OLA}^{a,\infty}_{\tilde\fb}$ is an equivalence between the categories $\mathcal{OLA}^{a,\infty}_{\tilde\fb}$ and
  $\mathcal{OLA}_{\fb}$. Since $\fg_{\tilde P}$ is isomorphic to $\fg$, the Borel subalgebra $\fb'\subset \fg$ can be chosen as the image of $\tilde\fb$ under
  an isomorphism $\fg_{\tilde P}\backsimeq\fg$, and the statement follows.
\end{proof}
\begin{corollary}\label{1cor}
If $\fb$ is an arbitrary splitting Borel subalgebra of $\fg$, there exists a perfect Borel  subalgebra $\fb'\subset\fg$ such that the category  $\mathcal{OLA}_{\fb}$ is equivalent to a full subcategory of $\mathcal{OLA}_{\fb'}$.
\end{corollary}
\begin{proof}
Follows from Theorem \ref{qr}, Proposition \ref{prop:fin} and Proposition \ref{prop:onesided}.
\end{proof}

\section{$\mathcal{OLA}$: simple and parabolically induced modules}

Corollary \ref{1cor} suggests that it makes sense to restrict our study of the category $\mathcal{OLA}_{\fb}$ to the case when $\fb$ is a fixed perfect
Borel subalgebra.  In the rest of the paper we do this and  write $\mathcal{OLA}$, omitting the subscript $\fb$. Furthermore, Corollary \ref{cor:BD} allows us to disregard the case $\Delta=B_\infty$ and assume that $\Delta=D_\infty$ for $\fg=\mathfrak{o}(\infty)$.

 By $\bar \fb =\fh\crplus\bar \fn$ we denote the opposite Borel subalgebra, $\fb\cap\bar\fb=\fh$. In addition, for $\fg=\mathfrak{sl}(\infty)$ we identify the ordered set $\tilde I$ with $\mathbb{Z}_{>0}\sqcup\mathbb{Z}_{<0}$ where $i<-j$ for $i$, $j\in \mathbb{Z}_{>0}$, so that $S_{min}=\{\varepsilon_i| i\in \mathbb{Z}_{>0}\}$, $S_{max}=\{\varepsilon_i| i\in \mathbb{Z}_{<0}\}$. For $\fg=\mathfrak{o}(\infty)$, $\mathfrak{sp}(\infty)$ we identify $S_{min}$ with $\mathbb{Z}_{>0}$,  and write $S_{min}=\{\varepsilon_i| i\in \mathbb{Z}_{>0}\}$; then $S_{max}=- S_{min}$ $=\{-\varepsilon_i| i\in \mathbb{Z}_{>0}\}$.

Let $\fk_n$ be  the centralizer of $\fg_{n,n}$ in $\fg$. Note that $\fg_{n,n}\simeq \fg$ and
  $$\fk_n\simeq\begin{cases}\mathfrak{sl}(2n)\,\text{for}\, \fg=\mathfrak{sl}(\infty)\\ \mathfrak{o}(2n)\,\text{for}\, \fg=\mathfrak{o}(\infty)\\ \mathfrak{sp}(2n)\,\text{for}\, \fg=\mathfrak{sp}(\infty).\end{cases}$$

Next, fix compatible nodegenerate invariant forms on $\fk_n$ which define a nondegenerate invariant form $(\cdot,\cdot)$ on $\fg$. We will use the same notation when considering $(\cdot, \cdot)$ as a form on $\fh^*$.

In what follows we  will  use the family of parabolic subalgebras of $\fg$ 
$$\fp_n:=\fb+\fg_{n,n}$$
with reductive parts $\fl_n=\fh+\fg_{n,n}$. By $\bar\fp_n$ we denote the parabolic subalgebra opposite to $\fp_n$, $\fp_n\cap \bar \fp_n=\fl_n$.
Furthermore, we define $\fm_n$ as the nilpotent ideal of $\fp_n$ such that $\fp_n=\fl_n\crplus\fm_n$. The space of $\fg_{n,n}$-invariants $ \fm_n^{\fg_{n,n}}$ is finite
dimensional, and the decomposition of $\fg_{n,n}$-modules $\fm_n=\fr_n\oplus \fm_n^{\fg_{n,n}}$ defines $\fr_n\subset\fm_n$. 

In addition, we introduce the  subalgebras $\fs\subset\fg$ and $\fs_n\subset\fk_n$ by setting
$$\fs:=\fh\oplus\bigoplus_{\alpha\in\Delta_{fin}}\fg_{\alpha},\quad\fs_n:=\fs\cap\fk_n,$$
where $\Delta_{fin}$ stands for the $\fb$-finite roots.
We have
$$\fs\simeq\begin{cases}\mathfrak{sl}(\infty)\oplus\mathfrak{sl}(\infty)\oplus\mathbb C\,\text{for}\, \fg=\mathfrak{sl}(\infty)\\ \mathfrak{gl}(\infty)\,\text{for}\, \fg=\mathfrak{o}(\infty), \mathfrak{sp}(\infty)\end{cases}.$$
and
$$\fs_n\simeq\begin{cases}\mathfrak{sl}(n)\oplus\mathfrak{sl}(n)\oplus\mathbb C\,\text{for}\, \fg=\mathfrak{sl}(\infty)\\ \mathfrak{gl}(n)\,\text{for}\, \fg=\mathfrak{o}(\infty), \mathfrak{sp}(\infty)\end{cases}.$$
  Note that $\fh_n:=\fh\cap\fk_n$ is a Cartan subalgebra of $\fk_n$ as well as of $\fs_n$.

For $\fg=\mathfrak o(\infty)$, $\mathfrak{sp}(\infty)$ we denote  by $V_n$ the natural $\mathfrak{gl}_n = \fs_n$-module. For $\fg=\mathfrak{sl}(\infty)$ we set $V_n=V_n^L\oplus V_n^R$, where $\operatorname{supp} V_n^L= \{\varepsilon_i| 1\leq i \leq n\}$ and $\operatorname{supp} V_n^R= \{\varepsilon_i| -n\leq i \leq -1\}$. Then there is canonical decomposition $$V= V_n\oplus \bar{V}_n$$ where $\bar{V}_n$ is the natural $\fg_{n,n}$-module (the notion of natural module makes sense for $\fg_{n,n}$ as $\fg_{n,n}$ is isomorphic to $\fg$). Moreover, we have the following isomorphism of $\fg_{n,n}$-modules:
$$\fr_n\simeq\begin{cases}\bar{V}_n^{\oplus n}\oplus (\bar V_n)_*^{\oplus n} \,\text{for}\, \fg=\mathfrak{sl}(\infty)\\ \bar V_n^{\oplus n}\,\text{for}\,\fg=\mathfrak{o}(\infty),\mathfrak{sp}(\infty)\end{cases}.$$

\subsection{Simple modules}\label{Simple modules} We start with the following lemma.
\begin{lemma} \label{fg} There exists a finite-dimensional $\operatorname{ad}(\fm_n^{\fg_{n,n}})$-stable subspace
 $\mathfrak u\subset\mathfrak r_{n}$ such that ${\bf S}(\mathfrak u)$ generates
${\bf S}(\fr_{n})$ as a module over $\mathfrak g_{n,n}$. 
\end{lemma} 
\begin{proof} Let $\fg=\mathfrak{sl}(\infty)$. Then
$${\bf S}(\fr_{n})= {\bf S}((\bar{V}_n)_*^{\oplus n}\oplus \bar{V}_n^{\oplus n})=\bigoplus_{\lambda, \mu} (\mathbb{S}_\lambda((\bar{V}_n)_*)\otimes \mathbb{S}_{\mu}(\bar{V}_n))^{\oplus c(\lambda,\mu)}$$
for some $c(\lambda,\mu)\in\mathbb{Z}_{\geq 0}$, where the summation is taken over all partitions $\lambda,\mu$ with at most $n$ parts. Recall that, by Lemma 4.1(a) in \cite{DPS}, the $\fg$-module $\mathbb S_\lambda((\bar{V}_n)_*)\otimes \mathbb S_\mu(\bar{V}_n)$ is generated by $\mathbb S_\lambda(Z'_n)\otimes \mathbb S_\mu(Z_n)$ for some $n$-dimensional subspaces $Z'_n \subset (\bar{V}_n)_*$ and $Z_n \subset \bar{V}_n$.
Therefore, ${\bf S}(\mathfrak r_n)$ is also generated by ${\bf S}(\mathfrak{u})$ for some finite-dimensional space $\mathfrak{u}\subset \fr_n$. As $\fm_n^{\fg_{n,n}}$ is finite dimensional and its elements act locally finitely on $\fr_n$, the subspace $\fu$ can clearly be chosen ad$(\fm_n^{\fg_{n,n}})$-stable.

In the orthogonal and symplectic case we have the decomposition
$${\bf S}(\fr_{n})= {\bf S}(\bar{V}_n^{\oplus n})=\bigoplus_{\lambda} (\mathbb S_\lambda(\bar{V}_n))^{\oplus c(\lambda)},$$
for some $c(\lambda)\in\mathbb{Z}_{\geq 0}$, where $\lambda$ runs over all partitions with at most $n$ parts. Here, application of Lemma 4.1(b) from \cite{DPS} leads to the result.
\end{proof}

By $U(\cdot)$ we denote as usual the enveloping algebra of a Lie algebra, and $U^k(\cdot)$ stands for the k-th term of the PBW filtration on $U(\cdot)$.

\begin{proposition}\label{glnil} Let $M\in\mathcal{OLA}$ and   $0\neq v\in M$ satisfy $\fg_{n,n}v=0$.
  
(a) There exists $m\in\mathbb Z_{>0}$ such that $\fm_n^mv=0$.
  
(b) $(U(\fm_{n})v)^{\fm_n}\neq 0$.
\end{proposition}
\begin{proof} Let us prove (a). Since every element of $\fm_n$ acts locally nilpotently on $M$, it suffices to check that
  $U^k(\fm_n)v=U(\fm_n)v$ for sufficiently large $k$. Then $m$ can be chosen as $k+1$. Note that $\fm_n$ is  ad$(\fg_{n,n})$-stable, therefore $U^k(\fm_n)v$ is also ad$(\fg_{n,n})$-stable. Choose $\mathfrak{u}\subset\fr_n$ as in Lemma \ref{fg} and set $\mathfrak{a}:=\mathfrak{u}{\clplus} \fm_n^{\fg_{n,n}}$.
  Since $\mathfrak{a}$ is a nilpotent finite-dimensional Lie algebra we have  $U^k(\mathfrak{a})v=U(\mathfrak{a})v$ for sufficiently large $k$. On the other hand, $U^k(\mathfrak{a})$ 
  (respectively, $U(\mathfrak{a})$) generates  $U^k(\fm_n)$  (respectively, $U(\fm_n)$) as an adjoint $\fg_{n,n}$-module.  This implies that the $\fg_{n,n}$-submodules of $U(\fg)v$ generated respectively by $U^k(a)v$ and $U(a)v$ coincide. {As} these modules equal respectively
  $U^k(\fm_n)v$ and  $U(\fm_n)v$, we obtain  $U^k(\fm_n)v=U(\fm_n)v$.
 
(b) follows immediately from (a).
  \end{proof}

\begin{theorem}\label{simple} Let $L\in\mathcal{OLA}$ be a simple object. Then there exist $n\in\mathbb Z_{\geq 0}$ and
a weight $\lambda\in \fh^*$, such that $\lambda|_{\fh\cap \fg_{n,n}}=0$ and $L$ is isomorphic to the unique simple quotient of the induced module 
$\operatorname{Ind}^{\mathfrak g}_{\mathfrak p_n} \mathbb C_\lambda$. In particular, $L$ is a highest weight module with highest weight $\lambda$, and we denote it
by $L(\lambda)$.
\end{theorem}
\begin{proof} The large annihilator condition ensures that for any $v\in L$ we have $\fg_{k,k}v=0$ for some $k\in \mathbb{Z}_{\geq 0}$. Therefore Proposition \ref{glnil}(b) implies $L^{\fm_k}\neq 0$ for some $k$.
Since $L$ is simple,  $L^{\fm_k}$ is a simple $\fl_k$-module. Moreover, as a $\mathfrak{g}_{k,k}$-module,  $L^{\fm_k}$ is integrable and satisfies the large annihilator 
condition. Hence, by Theorem 4.2 in \cite{DPS}, $L^{\fm_k}$ has a highest weight vector $u$ with respect to the Borel 
subalgebra $\fb\cap\fg_{k,k}$. Since $\fn = \fm_k \clplus (\fn \cap \fg_{k,k})$, we obtain $\fn u=0$. Denote by $\lambda$ the weight of $u$. 
By the large annihilator condition, there exists $n\geq k$ such that $\fg_{n,n} u=0$. This implies that $\mathbb C u$ is a one-dimensional
$\fp_n$-module isomorphic to $\mathbb C_\lambda$.
Then by Frobenius reciprocity $L$ is isomorphic a quotient of $\operatorname{Ind}^{\mathfrak g}_{\mathfrak p_n} \mathbb C_\lambda$. 
\end{proof}

In what follows, we call a weight $\lambda\in \fh^*$ \emph{eligible} if $\lambda|_{\fh\cap\fg_{n,n}}=0$ for some $n \in \mathbb{Z}_{\geq 0}$. The set of eligible weights coincides with the subspace $\langle \tilde I\rangle_\mathbb C\subset \fh^*$. Note that for $\fg=\mathfrak{sl}(\infty)$ an eligible weight $\lambda$ has the form $\lambda^L+\lambda^R$ for  uniquely determined eligible weights $\lambda^L:=\sum_{i\in \mathbb Z_{>0}}\lambda_i\varepsilon_i$ and $\lambda^R:=\sum_{i\in \mathbb{Z}_{<0}}\lambda_i\varepsilon_i$ (recall that in this case $S_{min}=\{\varepsilon_i|i\in \mathbb{Z}_{>0}\}$, $S_{max}=\{\varepsilon_i|i\in \mathbb{Z}_{<0}\}$ ). Furthermore, Theorem 4.3 claims that any simple object of $\mathcal{OLA}$ is a $\fb$-highest weight module with an eligible highest weight.

A weight $\lambda$ is $\fb$-\emph{dominant} if $2\frac{(\lambda,\alpha)}{(\alpha,\alpha)}\in \mathbb{Z}_{\geq0}$ for all $\alpha\in\Delta^+ $.
We observe that for $\fg=\mathfrak{sl}(\infty)$  an eligible weight $\lambda$ is $\fb$-dominant iff $\lambda^1:=(\lambda^L_1\geq\dots\geq\lambda^L_l\geq\dots)$ and $\lambda^2:=(-\lambda^R_{-1}\geq\dots\geq -\lambda^R_{-r}\geq\dots)$ are partitions where $\lambda_i\in\mathbb Z_{\geq 0}$ for $i\in \mathbb Z_{>0}$,
$\lambda_i\in\mathbb Z_{\leq 0}$ for $i\in \mathbb Z_{<0}$.
For $\fg=\mathfrak{o}(\infty)$, $\mathfrak{sp}(\infty)$ an eligible weight $\lambda$ is $\fb$-dominant   iff 
$\lambda=(\lambda_1,\dots,\lambda_k,\dots)$ is itself a partition.
In \cite{DPS} the simple modules of the category $\mathbb{T}_{\fg}$  are parametrized as $V_{(\lambda^1,\lambda^2)}$ for $\fg=\mathfrak{sl}(\infty)$, and as 
$V_{\lambda}$ for $\fg = \mathfrak{o}(\infty), \mathfrak{sp}(\infty)$, where $\lambda^1, \lambda^2, \lambda$ are partitions. As we pointed out in the introduction, $\mathbb{T}_{\fg}$ is a full subcategory of $\mathcal{OLA}$, and the simple modules $V_{(\lambda^1,\lambda^2)}$ are denoted in the present paper as $L(\lambda)$ where $\lambda=(\lambda^1, \lambda^2)$ for $\fg=\mathfrak{sl}(\infty)$, and where  $\lambda$ is considered both as an eligible weight and as a partition for $\fg=\mathfrak{o}(\infty), \mathfrak{sp}(\infty)$. 

\subsection{Parabolically induced modules $\operatorname{Ind}^{\mathfrak g}_{\mathfrak p_n} \mathbb C_\lambda$ }
For an eligible weight $\lambda$, we set
$$M_n(\lambda):=\operatorname{Ind}^{\mathfrak g}_{\mathfrak p_n} \mathbb C_\lambda,$$
where we always assume that $n$ is large enough to ensure that $\mathbb C_\lambda$ is a trivial $\fg_{n,n}$-module.
\begin{lemma}\label{lem:intquo} A nonzero integrable quotient of  $M_n(\lambda)$ is  simple.
\end{lemma}
\begin{proof} Since $\fb\subset \fp_n$, any quotient of $M_n(\lambda)$ is a $\fb$-highest weight module. An integrable quotient  of $M_n(\lambda)$ is an object $\mathbb T_{\mathfrak g}$, and is hence isomorphic to a submodule of a finite direct sum $\bigoplus_{i}V^{\otimes n_i}\otimes (V_*)^{\otimes m_i} $ for some $m,n \in \mathbb{Z}_{\geq 0}$. (In the case $\fg=\mathfrak{o}(\infty), \mathfrak{sp}(\infty)$ we assume $V=V_*$.) However, the explicit form of the socle filtration of $\bigoplus_{i}V^{\otimes n_i}\otimes (V_*)^{\otimes m_i}$, see \cite{PStyr} or \cite{DPS}, implies that a $\fb$-highest weight submodule of $\bigoplus_{i}V^{\otimes n_i}\otimes (V_*)^{\otimes m_i}$ is necessarily simple.
\end{proof}

\begin{lemma}\label{par}  The module $M_n(\lambda)$, considered as an   ${\fl_n}$-module,  has a decomposition 
$\bigoplus M_i$ such that each $M_i$ is a finite-length
$\fl_n$-module. Moreover, the Jordan-H\"older  multiplicity of every simple $\fl_n$-module in 
$M_n(\lambda)$ is finite.
\end{lemma}
\begin{proof} We have an isomorphism of $\fl_n$-modules
  $$M_n(\lambda)\simeq {\bf S}(\bar\fm_n)\otimes\mathbb C_\lambda$$
where  $\bar\fm_n$ is the nilpotent ideal such that $\bar\fp_n=\fl_n \crplus \bar\fm_n$. Let $z \in \fl_n$ be a central element which defines a finite $\mathbb Z_{<0}$-grading on $\bar\fm_n$. Consider the decomposition
$M_n(\lambda)=\bigoplus_i M_i$  into  $\text{ad}z$-eigenspaces.
Then every $M_i$ is isomorphic to a submodule in $(\oplus_{i-k<j<i+k}{\bf S}^j(\bar\fm_n))\otimes \mathbb C_\lambda$ for sufficiently large $k$.
Thus $M_i$ is a finite-length $\fl_n$-module, and the statement follows.
\end{proof}
\begin{corollary}\label{JH} There is a descending  filtration
  $$M_n(\lambda)= (M_n(\lambda))_0
  \supset (M_n(\lambda))_1\supset\dots\supset (M_n(\lambda))_{{i}}\supset\dots$$
  such that $\bigcap_i  (M_n(\lambda))_i=0$ and
  $ (M_n(\lambda))_i/ (M_n(\lambda))_{i+1}$ is simple for
  all $i\geq 0$. Furthermore, the subquotient multiplicity $[M_n(\lambda):L(\mu)]$ of any simple module $L(\mu)$
  defined by such a filtration is finite and does not depend on the choice of a filtration.
\end{corollary}
\begin{proof} Lemma \ref{par} implies the statement if we consider $M_n(\lambda)$
  as a module over $\fl$. Hence, the statement holds also for $\fg\supset\fl$.
\end{proof}

\subsection{Jordan--Hoelder multiplicities for parabolically induced modules} 
Consider the functor $$\Phi_n:\mathcal {OLA}\to\tilde{\mathcal O}_{\fk_n}, \, \Phi_n(M):=M^{\fg_{n,n}},$$ 
$\tilde{\mathcal O}_{\fk_n}$ being the inductive completion of the BGG category $\mathcal{O}$ for the finite-dimensional Lie algebra $\fk_n$.
The large annihilator condition ensures that for any $M\in \mathcal{OLA}$ 
$$M=\lim_{\longrightarrow}\Phi_n(M).$$
\begin{lemma}\label{inv1} For $m\geq n$ we have an isomorphism of $\fk_m$-modules
  $$\Phi_m(M_n(\lambda))\simeq\operatorname{Ind}^{\fk_m}_{\fp_n\cap\fk_m}\mathbb C_\lambda.$$  
\end{lemma}
\begin{proof} Note that the result of application of $\Phi_m$ depends only on the restriction to $\fg_{m,m}$. Therefore, the statement follows from the isomorphism of $\fg_{m,m}$-modules
  $$M_n(\lambda)\simeq {\bf S}(\bar\fm_n)\otimes\mathbb C_\lambda$$ and the fact that
  $$ {\bf S}(\bar\fm_n)^{\fg_{m,m}}={\bf S}(\bar\fm_n\cap\fk_m).$$
  \end{proof}

  \begin{lemma}\label{phi} Let $M,N\in\mathcal{OLA}$ and $U(\fg)\Phi_n(M)=M$. Then the natural map
    $\Hom_{\fg}(M,N)\to \Hom_{\fk_n}(\Phi_n(M),\Phi_n(N)) $ is injective.
  \end{lemma}
  \begin{proof} Straightforward.
  \end{proof}
  \begin{corollary}\label{cor:phi} Let $m\geq n$. The natural map
    $$\Hom_{\fg}(M_n(\lambda),N)\to \Hom_{\fk_m}(\operatorname{Ind}^{\fk_m}_{\fp_n\cap\fk_m}\mathbb C_\lambda,\Phi_m(N)) $$
    is injective.
    \end{corollary}
\begin{lemma}\label{order} If
  $\Hom_{\fg}(M_n(\lambda),M_m(\mu))\neq 0$ for $\lambda\neq\mu$, then $\lambda-\mu$
  is a sum of positive  finite roots.
\end{lemma}
\begin{proof} By Corollary \ref{cor:phi}, $\Hom_{\fg}(M_n(\lambda),M_m(\mu))\neq 0$ implies
$$\Hom_{\fk_p}(\operatorname{Ind}^{\fk_p}_{\fp_n\cap\fk_p}\mathbb C_\lambda,\operatorname{Ind}^{\fk_p}_{\fp_m\cap\fk_p}\mathbb C_\mu)\neq 0$$
  for all $p\geq m,n$.
Consequently $\mu|_{\fh_p}+\rho_p$ and $\lambda|_{\fh_p}+\rho_p$ lie in one orbit of the Weyl group $W_p$ of  $\fk_p$, where $\rho_p$ denotes the half-sum of positive roots of $\fk_p$. In other words,
$$\lambda|_{\fh_p} - w_p(\mu|_{\fh_p})=w_p(\rho_p)-\rho_p$$
for some $w_p\in W_p$. When $p\to \infty$  the quantity $|(\lambda|_{\fh_p} - w_p(\mu|_{\fh_p}),\alpha)|$ remains bounded for any fixed $\alpha \in \Delta$ and any $w_p\in W_p$, while the quantity $|(w_p(\rho_p)-\rho_p,\alpha)|$ remains bounded if and only if $w_p$ is a product of reflections corresponding to simple roots of $\fk_p$ which are finite as roots of $\fg$. Therefore $w_p$ must have the latter property, and this implies the statement.
\end{proof}

\begin{lemma}\label{lem:JH} Let $L_{\fk_m}(\mu|_{\fh_m})$ denote a simple $\fk_m$-module with a $\fb\cap \fk_m$-highest weight vector of weight $\mu|_{\fh_m}$.
If $[M_n(\lambda):L(\mu)]\neq 0$, then   $[\Phi_m(M_n(\lambda)):L_{\fk_m}(\mu|_{\fh_m})]\neq 0$ for sufficiently large $m$.
\end{lemma}
\begin{proof} If $[M_n(\lambda):L(\mu)]\neq 0$, then there exists a nonzero vector $u\in M_n(\lambda)$ of weight $\mu$ and a submodule $X\subset M_n(\lambda)$
  such that $\fn u\in X$ and $u\notin X$. For all sufficiently large $m$, we have $u\in\Phi_m(M_n(\lambda))$. Then $(\fn\cap\fk_m)u\in \Phi_m(X)$. Therefore
  $[\Phi_m(M_n(\lambda)):L_{\fk_m}(\mu|_{\fh_m})]\neq 0$.
  \end{proof}

\begin{lemma}\label{order1} If $[M_n(\lambda):L(\mu)]\neq 0$ for $\lambda\neq\mu$, then $\lambda-\mu$ is a sum of positive  finite roots.
\end{lemma}
\begin{proof} The previous lemma implies $[\Phi_m(M_n(\lambda)):L_{\fk_m}(\mu|_{\fh_m})]\neq 0$ for all sufficiently large $m$. Therefore we can use the same
  argument as in
  the proof of Lemma \ref{order}.
\end{proof}

Let $\mathscr{W}$ be the group generated by all reflections with respect to the simple roots of our fixed Borel subalgebra $\fb$. Then
$\mathscr{W}\simeq \mathscr{S}_\infty\times \mathscr{S}_\infty$ for $\fg=\mathfrak{sl}(\infty)$
and $\mathscr{W}\simeq \mathscr{S}_\infty$ for $\fg=\mathfrak{o}(\infty)$, $\mathfrak{sp}(\infty)$; here $\mathscr{S}_{\infty}$ denotes the infinite symmetric group. We fix $\rho\in \fh^*$ such that $2\frac{(\rho,\alpha)}{(\alpha,\alpha)}=1$ for any simple root $\alpha$.

We define a partial  order ${\leq}_{fin}$ on the set of eligible
weights  by setting $\mu\leq_{fin}\lambda$ if  $\mu=\lambda$ or $\lambda-\mu$ is a sum of positive simple roots and $(\lambda+\rho)=w(\mu+\rho)$ for some $w\in \mathscr{W}$.
This order is interval-finite. In fact, the following stronger property holds:
for any eligible weight
$\mu$, the set $$\mu_{fin}^+:=\{\lambda\,|\,\mu\leq_{fin}\lambda\}$$ is finite.

Lemmas  \ref{lem:JH} and \ref{order1}  imply the following.
\begin{corollary}\label{cor:mult} If $[M_n(\lambda):L(\mu)]\neq 0$, then $\mu\leq_{fin}\lambda$.
  \end{corollary}

\begin{lemma} \label{lem:stabilization}
  Given two eligible weights $\lambda$ and $\mu$, there exists $N \in \mathbb{Z}_{\geq 0}$ such that the multiplicity $[M_n(\lambda):L(\mu)]$ is 
constant  for $n>N$. We denote this constant multiplicity by $m(\lambda,\mu)$.
\end{lemma}
\begin{proof} Choose $N$ such that $\lambda-\mu$ is a sum of  roots of $\fk_{N,N}$. For $n>N$, consider the canonical surjection homomorphism 
  $\varphi:M_n(\lambda)\to M_N(\lambda)$. We have $\mu\notin\operatorname{supp}\Ker\varphi$. Hence $[\Ker\varphi:L(\mu)]=0$. This implies
  $[M_{n}(\lambda):L(\mu)]=[M_{N}(\lambda):L(\mu)]$.
\end{proof}
\begin{lemma}\label{parfl} The $\fg$-module  $M_n(\lambda)$ has finite length.
\end{lemma}
\begin{proof} We claim that there are finitely many weights $\mu$ for which $[M_n(\lambda):L(\mu)]\neq 0$.
  Indeed, $[M_n(\lambda):L(\mu)]\neq 0$  implies that $\fg[L({\mu})]\subset \fg_{n,n}$ and hence the restriction  of $\mu$ to ${\fh\cap\fg_{n,n}}$ is
  ${\fb\cap\fg_{n,n}}$-dominant. On the other hand, by Corollary \ref{cor:mult} $\mu=w(\lambda+\rho)-\rho$ for some $w\in \mathcal W$. This is possible only for finitely many
  $w$, and hence for finitely many $\mu$.
\end{proof}

The following lemma shows that the multiplicities $m(\lambda,\mu)$ can be expressed in terms of  Kazhdan-Lusztig multiplicities for the BGG category $\mathcal O_{\fs_n}$
of the reductive Lie algebra $\fs_n$ for sufficiently large $n$. 
\begin{lemma}\label{lem:KL} Let $\lambda,\mu$ be eligible weights such that $\mu\leq_{fin}\lambda$, and let
  $\lambda|_{\fh\cap\fg_{n,n}}=\mu|_{\fh\cap\fg_{n,n}}=0$ for some $n$. Then $$m(\lambda,\mu)=[M_{\fs_n}(\lambda|_{\fh_n}):L_{\fs_n}(\mu|_{\fh_n})]$$
  where   $M_{\fs_n}(\lambda|_{\fh_n})$ and
  $L_{\fs_n}(\mu|_{\fh_n})$ denote the respective Verma and  simple module over $\fs_n$.
\end{lemma}
\begin{proof} Consider the parabolic subalgebra $\fq_n=\fs_n+\fp_n$. Then
  $M_n(\lambda)\simeq \operatorname{Ind}^{\fg}_{\fq_n}M_{\fs_n}(\lambda|_{\fh_n})$. Since $\operatorname{Ind}^{\fg}_{\fq_n}$ is an exact functor, we have
  $m(\lambda,\mu)\geq [M_{\fs_n}(\lambda|_{\fh_n}):L_{\fs_n}(\mu|_{\fh_n})]$. Choose $h\in\fh$ such that $[h,\fs_n]=[h,\fg_{n,n}]=0$ and $\alpha(h)=1$
  for the simple roots $\alpha$ which  are not roots of $\fs_n\oplus\fg_{n,n}$. Then
  $L(\mu)^{h-\lambda(h)}\simeq L_{\fs_n}(\mu|_{\fh_n})$ and $M_n(\lambda)^{h-\lambda(h)}\simeq M_{\fs_n}(\lambda|_{\fh_n})$, the superscript indicating taking invariants. Hence $m(\lambda,\mu)\leq [M_{\fs_n}(\lambda|_{\fh_n}):L_{\fs_n}(\mu|_{\fh_n})]$.
The statement follows.
\end{proof}

\begin{proposition}\label{fl} Any finitely generated module in $\mathcal{OLA}$ has finite length.
\end{proposition}
\begin{proof} It suffices to check the statement for a cyclic module. Assume that $M$ is generated by some weight vector $v$ annihilated by $\fg_{n,n}$.
Then $\fm_n^m v=0$ for some $m$ by
   Proposition \ref{glnil} (a). Therefore $\operatorname{dim}U(\fm_n)v<\infty$, and 
 there is a finite filtration
  $\{(U(\fm_n)v)_i\}$ of $U(\fm_n)v$ such that every quotient $(U(\fm_n)v)_i/(U(\fm_n)v)_{i-1}$ is annihilated by $\fm_n$. Moreover,  $(U(\fm_n)v)_i/(U(\fm_n)v)_{i-1}$ is an
  object of the category $\mathbb T_{\fg_{n,n}}$.
  Hence one can refine this filtration of $U(\fm_n)v$ and obtain a finite filtration
  $$0\subset F_1\subset\dots\subset U(\fm_n)v,$$
  such that $F_i/F_{i-1}$ is a simple integrable $\fg_{n,n}$-module annihilated by $\fm_n$.

 Consider the induced filtration of $M$:
  $$0\subset U(\fg)F_1\subset\dots\subset U(\fg)v=M.$$
  Then $U(\fg)F_i/U(\fg)F_{i-1}$ is isomorphic to a quotient of the induced module $\operatorname{Ind}^\fg_{\fp_n}(F_i/F_{i-1})$, and the latter module is
  isomorphic to a quotient of
  $M_t(\lambda)$ for some $t>n$ and some $\lambda$. Since $M_t(\lambda)$ has finite length, the same is true for $U(\fg)F_i/U(\fg)F_{i-1}$, and thus for $M$.
\end{proof}

\begin{proposition}\label{soc} Any $M\in\mathcal{OLA}$ has an exhausting socle filtration.
\end{proposition}
\begin{proof} Any module is a union of finitely generated modules. By Proposition \ref{fl}  any finitely generated module in $\mathcal{OLA}$ has a finite exhausting socle filtration. The statement follows.
\end{proof}

\subsection{Canonical filtration on $\mathcal{OLA}$ } For an eligible weight $\lambda=\sum_{\varepsilon_i\in I}\lambda_i\varepsilon_i$ we set
$$d(\lambda)=\begin{cases}\frac{1}{2}(\sum_{i\in \mathbb{Z}_{>0}}\lambda_i-\sum_{j\in \mathbb{Z}_{<0}}\lambda_j)\,\,\text{if}\,\, \fg=\mathfrak{sl}(\infty)\\
  \frac{1}{2}\sum_{i\in \mathbb{Z}_{>0}}\lambda_i\,\,\text{if}\,\, \fg=\mathfrak{o}(\infty), \mathfrak{sp}(\infty).\end{cases}$$
Note that if $\lambda-\mu\in \langle\Delta^+\rangle_{\mathbb{Z}_{\geq 0}}$, then $d(\lambda)-d(\mu)\in\mathbb Z_{\geq 0}$.
\begin{lemma}\label{lem:ext} Assume $\operatorname{Ext}_{\mathcal{OLA}}^1(L(\lambda),L(\mu))\neq 0$. Then $d(\lambda)-d(\mu)\in\mathbb Z_{\leq 0}$.
\end{lemma}
\begin{proof}  Recall that if $M$ is   a  $\fg$-module and $\lambda\in\fh^*$, then $M_\lambda$ is the weight space of  weight  $\lambda$,  $$ M_{\lambda}:=\{m\in M| hm=\lambda(h)m \; \forall h\in \fh\}.$$       

    Consider a nonsplit exact sequence in $\mathcal{OLA}$
  $$0\to L(\mu)\to M\to L(\lambda)\to 0. $$
  Since $M$ is a $\fh$-semisimple, a standard argument shows that $\lambda\neq\mu$.

  We claim that that either  $\mu-\lambda\in  \mathbb \langle\Delta^+\rangle_{\mathbb{Z}_{> 0}}$ or
  $\lambda-\mu\in  \mathbb  \langle\Delta^+\rangle_{\mathbb{Z}_{> 0}}$. Indeed, assume $\lambda-\mu\not\in \mathbb \langle\Delta^+\rangle_{\mathbb{Z}_{> 0}}$. Then  the weight space $M_\mu$ must be a subspace of $U(\fb)M_\lambda$ as otherwise  the sequence would split. Therefore
  $\mu-\lambda\in  \mathbb \langle\Delta^+\rangle_{\mathbb{Z}_{> 0}}$.

  If  $\mu-\lambda\in  \mathbb \langle\Delta^+\rangle_{\mathbb{Z}_{> 0}}$, then $d(\lambda)-d(\mu)\in\mathbb Z_{\leq 0}$. If  $\lambda-\mu\in  \mathbb  \langle\Delta^+\rangle_{\mathbb{Z}_{> 0}}$, then $M$ is isomorphic to a quotient of $M_n(\lambda)$ for some $n$. Therefore
  $m(\lambda,\mu)\neq 0$. By Corollary \ref{cor:mult}  $\lambda-\mu$ is a sum of simple positive roots, and hence $d(\lambda)=d(\mu)$.
\end{proof}

\begin{corollary}\label{cor:ind} If $M\in\mathcal{OLA}$ is indecomposable, then $d(\nu)-d(\nu')\in \mathbb Z$ for any two weights of $M$.
\end{corollary}

We say that a simple module $L(\lambda)\in \mathcal{OLA}$ {\it has degree} $d$ if $d(\lambda)=d$.

\begin{lemma}\label{lem:filtration} Let $M\in\mathcal{OLA}$ have a simple constituent of degree $d\in\mathbb C$ an let the degree of every simple constituent of
  $M$ belong to $d+\mathbb{Z}_{\leq 0}$.
  Then there exists a unique submodule $N\subset M$ such that any simple constituent  of $N$ has degree $d$, and every simple constituent of
  $M/N$ has degree lying in   $d+\mathbb{Z}_{<0}$.
\end{lemma}
\begin{proof} Let $N$ be some maximal (possibly zero) submodule of $M$ whose simple subquotients have degree $d$. We claim that the degrees of all simple subquotients  of $M/N$ lie in $d+\mathbb{Z}_{<0}$. Indeed, if we assume the contrary, then at some level of the socle filtration of $M$ there is a simple constituent $L(\mu)$ of degree $d'=d+l$ for $l\in \mathbb{Z}_{<0}$,
  and there is a simple constituent $L(\lambda)$ of degree $d$  at the next level with a nontrivial extension of $L(\lambda)$ by $L(\mu)$. This contradicts
  Lemma \ref{lem:ext}.
\end{proof}
\begin{corollary}\label{cor:filtration}  Let $M\in\mathcal{OLA}$ satisfy the condition of Lemma \ref{lem:filtration}. Then $M$ has an exhausting canonical
  filtration
\begin{equation}\label{canonical}
  0=D_0(M)\subset D_1(M)\subset D_2(M)\subset\dots
\end{equation}
  such that all simple constituents of $D_{{i}}(M)/D_{{i}-1}(M)$ have degree $d-{i}+1$.
  \end{corollary}

We define $\mathcal{OLA}(\fs)$ as the category of $\fs$-modules which satisfy conditions (i)-(iii) of Section 3 for the Borel subalgebra $\fb\cap\fs$ of $\fs$  where $\fb$ is our fixed perfect Borel subalgebra of $\fg$. Next, we denote by $\mathcal{OLA}^d$ the full subcategory of $\mathcal {OLA}$ consisting of all objects whose simple constituents have degree $d$. Obviously,  
$\mathcal{OLA}^d$ is a Serre subcategory of $\mathcal {OLA}$.  For
any $M\in \mathcal{OLA}^d$ we set
$$M^+:=\bigoplus_{d(\mu)=d} M_\mu.$$
Then clearly $M^+$ is an object of $\mathcal{OLA}(\fs)$. Furthermore $(\cdot)^+: \mathcal{OLA}^d\to \mathcal{OLA}(\fs)$ is an exact faithful functor.
\begin{lemma}\label{lem:multred} For any objects $M$ and $L(\lambda)$ of $\mathcal{OLA}^d$, the multiplicity $[M:L(\lambda)]$ equals the multiplicity
  $[M^+:L_\fs(\lambda)]$ in $\mathcal{OLA}(\fs)$, $L_{\fs}(\lambda)$ being a simple $\fs$-module with $\fb\cap\fs$-highest weight $\lambda$.
\end{lemma}
\begin{proof} The proof is similar to the proof of Lemma \ref{lem:KL} and we leave it to the reader.
  \end{proof}

\section{$\mathcal{OLA}$ as a highest weight category}
In this section we show that $\mathcal{OLA}$ is a highest weight category according to Definition 3.1 of \cite{CPS}. In particular, this requires introducing standard objects parametrized by the eligible weights, as well as specifying an interval-finite partial order on  eligible  weights. 

\subsection{Standard objects} 

Consider the endofunctor $\Phi$ in the category $\fg$-$\mathrm{mod}$
$$\Phi(M):=\lim_{\longrightarrow} \Phi_n(M),\;\Phi_n(M):=M^{\fg_{n,n}}.$$
The restriction of $\Phi$ to $\mathcal{OLA}$ is the identity functor.

Recall also that, if $\Gamma_{\fh}(M)$ stands for the largest $\fh$-semisimple submodule of a $\fg$-module $M$, then $\Gamma_{\fh}$ is a well-defined endofunctor
on the category $\fg$-mod.

 Let now $M$ be a $\mathfrak g$-module such that the elements of $\mathfrak{n}$ act locally nilpotently on $\Gamma_{\fh}(M)$. Then $\Phi\circ\Gamma_{\mathfrak h}(M)$ is an object of $\mathcal{OLA}$, and
for any $X$ in $\mathcal{OLA}$ we have a canonical isomorphism 
\begin{equation}\label{tag}
\operatorname{Hom}_{\mathfrak g}(X,\Phi\circ \Gamma_{\mathfrak h} (M))=\operatorname{Hom}_{\mathfrak g}(X,M).
\end{equation}

Let $\operatorname{Ext}_{\fg,\fh}^i$ denote the ext-group in the category $\mathcal C_{\fg,\fh}$ of $\fg$-modules semisimple over $\fh$. As $\mathcal{OLA}$ is clearly
a Serre subcategory in $\mathcal C_{\fg,\fh}$, the equality
\begin{equation}\label{tag1}
  \operatorname{Ext}^1_{\fg,\fh}(M,N)=\operatorname{Ext}^1_{\mathcal{OLA}}(M,N)
  \end{equation}
 holds for any two objects $M,N$ of $\mathcal{OLA}$.
Moreover, if $X$ is an object of $\mathcal C_{\fg,\fh}$ with locally nilpotent action of the elements of $\mathfrak n$ and $N=\Phi  (X)$, we have
$\operatorname{Hom}_{\fg}(M,X/N)=0$ and hence an embedding
\begin{equation}\label{tag2}
\operatorname{Ext}^1_{\fg,\fh}(M,N)\hookrightarrow\operatorname{Ext}^1_{\fg,\fh}(M,X).
\end{equation}

For any eligible weight  $\lambda\in\langle\tilde I\rangle_{\mathbb C}$ let
$$\tilde W(\lambda):=\Gamma_{\mathfrak h}(\operatorname{Coind}^{\mathfrak g}_{\bar{\mathfrak b}}\mathbb C_\lambda).$$

We define the {\it standard object} $W(\lambda)$ by setting
$W(\lambda):=\Phi (\tilde W(\lambda) )$.
Since the elements of $\fn$ act locally nilpotently on $\tilde W(\lambda)$,
we conclude that $W(\lambda)$ is an object in $\mathcal{OLA}$.

\begin{lemma}\label{lem:coVerma} 
  
(a) The $\fg$-module $W(\lambda)$ is indecomposable with simple socle $L(\lambda)$;

(b) $\operatorname{dim}\operatorname{Hom}_{\fg}(M_n(\lambda),W(\mu))=\delta_{\lambda,\mu}$ for  sufficiently large $n$;

(c) $\operatorname{Ext}_{\mathcal{OLA}}^1(M_n(\lambda),W(\mu))=0$ for sufficiently large $n$.
\end{lemma}
\begin{proof} As we already pointed out, the  elements of $\mathfrak{n}$ act locally nilpotently on           $ \tilde W(\lambda)$. Therefore,            by (\ref{tag}) and Frobenius reciprocity we have
$$\Hom_{\mathfrak g}(L(\lambda),W(\mu))=
\operatorname{Hom}_{\mathfrak g}(L(\lambda),\operatorname{Coind}^{\mathfrak g}_{\bar{\mathfrak b}}\mathbb C_\mu)=
\operatorname{Hom}_{\bar{\mathfrak b}}(L(\lambda),\mathbb C_\mu).$$
Now (a) follows from the isomorphism of $\bar\fb$-modules  $L(\lambda)/\bar{\fb}L(\lambda)\simeq \mathbb C_\lambda$.

Let us prove (b). We have
$$\operatorname{Hom}_{\mathfrak g}(M_n(\lambda),W(\mu))=
\operatorname{Hom}_{\mathfrak g}(M_n(\lambda),\operatorname{Coind}^{\mathfrak g}_{\bar{\mathfrak b}}\mathbb C_\mu)=
\operatorname{Hom}_{\bar{\mathfrak b}}(M_n(\lambda),\mathbb C_\mu),$$
and (b) follows from the isomorphism of $\bar{\fb}$-modules $M_n(\lambda)/\bar{\fn}M_n(\lambda)\simeq \mathbb C_\lambda$.

 Next, we  prove (c). By (\ref{tag1}) and (\ref{tag2}) it suffices to show that
$$\operatorname{Ext}_{\fg,\fh}^1(M_n(\lambda),\tilde W(\lambda))=0.$$

We use Shapiro's Lemma:
$$\operatorname{Ext}^1_{\fg, \fh}(M_n(\lambda),\tilde W(\lambda))=
\operatorname{Ext}^1_{\bar{\mathfrak{b}}, \fh}(M_n(\lambda),\mathbb C_\mu).$$
Since $M_n(\lambda)$ is free over the nilpotent ideal $\bar{\mathfrak m}_n$, we have 
$\operatorname{Ext}^1_{\fh+\bar{\fm}_n,\fh}(M_n(\lambda),\mathbb C_\mu)=0$.
Therefore
$$\operatorname{Ext}^1_{\bar{\mathfrak b},\fh}(M_n(\lambda),\mathbb C_\mu)=\operatorname{Ext}^1_{\bar{\mathfrak b}\cap{\fg_{n,n}},\fh\cap\fg_{n,n}}(\mathbb C_\lambda,\mathbb C_\mu).$$
For sufficiently large $n$, we have $\lambda|_{\fh\cap\fg_{n,n}}=\mu|_{\fh\cap\fg_{n,n}}=0$. This implies
$$\operatorname{Ext}^1_{\bar{\mathfrak b}\cap{\fg_{n,n}},\fh\cap\fg_{n,n}}(\mathbb C_\lambda,\mathbb C_\mu)=
\operatorname{Ext}^1_{\bar{\mathfrak b}\cap{\fg_{n,n}},\fh\cap\fg_{n,n}}(\mathbb C,\mathbb C)=0.$$
\end{proof}

\begin{lemma}\label{extst}  If $\Ext_{\mathcal{OLA}}^1(L(\lambda),W(\mu))\neq 0$ or $\Ext_{\fg,\fh}^1(L(\lambda),\tilde W(\mu))\neq 0$, then $\mu<_{fin}\lambda$. 
\end{lemma}
\begin{proof}   Claim (c) of
  Lemma \ref{lem:coVerma}
  implies the existence of the surjective map  $$ \Hom_{\fg}(N(\lambda),W(\mu))\to\Ext_{\mathcal{OLA}}^1(L(\lambda),W(\mu))$$ where $N(\lambda)$ is the kernel of the canonical projection
  $M_n(\lambda)\to L(\lambda)$. The $\mathfrak{g}$-module $N(\lambda)$ has finite length and all simple constituents $L(\nu)$ of $N(\lambda)$ satisfy $\nu<_{fin}\lambda$.  Hence $\mu<_{fin}\lambda$. The statement for
  $\tilde W(\mu)$ is similar.
\end{proof}

\begin{corollary}\label{cor:stinj} If $\lambda_{fin}^+=\{\lambda\}$, then $W(\lambda)$ is injective in $\mathcal{OLA}$.
\end{corollary}
\begin{proof} By Proposition \ref{fl}, it suffices to check that $\Ext^1_{\mathcal{OLA}}(L(\mu),W(\lambda))=0$ for every eligible weight $\mu$. Thus, the statement
  is an immediate corollary of Lemma \ref{extst}.
  \end{proof}

\subsection{Injective objects}
Let us prove now that $\mathcal {OLA}$ has enough injective objects. Recall that $\fs$ denotes the subalgebra generated by $\mathfrak{h}$ and by all root spaces corresponding to  finite roots. Let
$L_{\fs}(\mu)$ be the simple $\mathfrak{b}\cap\mathfrak{s}$-highest weight module in the category $\bar{ \mathcal{O}_\fs}$ studied in \cite{Nam}. Since $\mu$ is almost dominant, $L_{\fs}(\mu)$
has an (indecomposable) injective envelope $I_{\fs}(\mu)$, see \cite{Nam}. Furthermore, let
$$\tilde W_{\fs}(\nu):=\Gamma_\fh(\operatorname{Coind}^\fs_{\fs\cap\bar\fb}\mathbb C_{\mu}).$$
It follows from \cite{Nam} that $I_{\fs}(\mu)$ has a finite filtration
  $$0=I_\fs(\mu)^0\subset I_\fs(\mu)^1\subset\dots\subset I_\fs(\mu)^k=I_\fs(\mu),$$
  such that $I_\fs(\mu)^i/I_\fs(\mu)^{i-1}\simeq \tilde W_\fs(\mu_i)$ with $\mu_1=\mu$ and $\mu_i>_{fin}\mu$ for $i>1$.

Set $\bar\fp:=\bar\fb+\fs$ and 
$$\tilde I(\mu):=\Gamma_\fh(\operatorname{Coind}^\fg_{\fs\cap\bar\fb}I_\fs(\mu)).$$
Since $\tilde W(\nu)\simeq \Gamma_\fh(\operatorname{Coind}^\fg_{\bar\fp}\mathbb C_\nu)$, we obtain that $\tilde I(\mu)$ has a finite filtration
\begin{equation}\label{tagfilt}
  0=\tilde I(\mu)^0\subset \tilde I(\mu)^1\subset\dots\subset \tilde I(\mu)^k=\tilde I(\mu),
  \end{equation}
  such that $\tilde I(\mu)^i/\tilde I(\mu)^{i-1}\simeq \tilde W(\mu_i)$ with $\mu_1=\mu$ and $\mu_i>_{fin}\mu$ for $i>1$.

Now{,} consider  $L(\lambda) $  for an arbitrary eligible weight. If $\lambda\notin\mu_{fin}^+$  then
$\Ext^1_{\fg,\fh}(L(\lambda),\tilde I(\mu))=0$ by Lemma \ref{extst}, while $\lambda\in \mu^+_{fin}$ implies $d(\lambda)=d(\mu)$. Therefore Shapiro's lemma implies
$$\Ext^1_{\fg,\fh}(L(\lambda),\tilde I(\mu))=\Ext^1_{\bar\fp,\fh}(L(\lambda),I_\fs(\mu)).$$
Furthermore, there is an isomorphism of  $\fs$-modules $L(\lambda)=L_{\fs}(\lambda)\oplus \bar \fr L(\lambda)$ where $\bar\fr$ is the nil-radical of $\bar\fp$. We have  $\Ext^1_{\bar\fp,\fh}(\bar\fr L(\lambda), I_{\fs}(\mu))=0$ as $d(\nu)<d(\mu)$  for any
weight $\nu$ of $\bar\fr L(\lambda)$. Consequently,
$$\Ext^1_{\bar\fp,\fh}(L(\lambda),I_\fs(\mu))=\Ext^1_{\fs,\fh}(L_{\fs}(\lambda),I_\fs(\mu))=0.$$
As a result, we obtain  $\Ext^1_{\fg,\fh}(L(\lambda),\tilde I(\mu))=0$ for any $\lambda$, and hence $I(\mu):=\Phi{(} \tilde I(\mu){)}$ is an injective object in $\mathcal{OLA}$
with socle $L(\mu)$.

\begin{proposition}\label{filtration} For any eligible weight $\mu$, the injective module $I(\mu)$ admits a finite filtration
  $$0=I(\mu)^0\subset I(\mu)^1\subset\dots\subset I(\mu)^k=I(\mu),$$
  such that $I(\mu)^i/I(\mu)^{i-1}\simeq W(\mu_i)$ with $\mu_1=\mu$ and $\mu_i>_{fin}\mu$ for $i>1$.
\end{proposition}
\begin{proof} The idea is  to apply  $\Phi$ to (\ref{tagfilt}). If $n$ is sufficiently large, there is an isomorphism of $\fg_{n,n}$-modules:
  $$\tilde W(\mu_i)\simeq \Gamma_\fh(\Hom_{\mathbb C}({\bf S}(\fb/(\fg_{n,n}\cap\fb)), \tilde W_{\fg_{n,n}}(0)))$$
  where $ \tilde W_{\fg_{n,n}}(0)$ is the obvious analogue of $\tilde W(0)$. Moreover, ${\bf S}(\fb/(\fg_{n,n}\cap\fb))$ is an object of $\tilde {\mathbb T}_{\fg_{n,n}}$.
  Since $\Ext^1_{\fg_{n,n},\fh_{n,n}}(L,\tilde W_{\fg_{n,n}}(0))$ for any object $L$ of $\mathcal{OLA}_{\fg_{n,n}}$, we get
    $\Ext^1_{\fg_{n,n},\fh_{n,n}}(\mathbb C,\tilde W(\mu_i))=0$. Hence $\Phi_n=\Hom_{\fg_{n,n}}(\mathbb C,\cdot)$ induces {a} filtration  
$$ 0=\Phi_n(\tilde I(\mu)^0)\subset \Phi_n{(}\tilde I(\mu)^1{)}\subset\dots\subset \Phi_n{(}\tilde I(\mu)^k{)}=\Phi_n{(}\tilde I(\mu){)},$$
such that $\Phi_n{(}\tilde I(\mu)^i{)}/\Phi_n{(}\tilde I(\mu)^{i-1}{)}\simeq  \Phi_n{(}\tilde W(\mu_i){)}$ with $\mu_1=\mu$ and $\mu_i>_{fin}\mu$ for $i>1$.
The statement follows by passing to the direct limit.
\end{proof}

\begin{proposition}\label{resolution} For any $\lambda\in\mathbb C\tilde I$, the module $W(\lambda)$ has a finite injective resolution
  $R^{\cdot}(\lambda)$ of length not greater
  than $|\lambda_{{fin}}^+|$ and satisfying the following properties:
\begin{enumerate}
\item if $I(\mu)$ appears in $R^{\cdot}(\lambda)$ then $\mu\geq_{fin}\lambda$;
\item the multiplicity of $I(\lambda)$ in  $R^{\cdot}(\lambda)$ equals $1$;
  \item the multiplicity of $I(\mu)$ in  $R^{\cdot}(\lambda)$ is finite for every $\mu$.
\end{enumerate}
\end{proposition}
\begin{proof} Immediate consequence of Proposition \ref{filtration}.
  \end{proof}

  \begin{corollary}\label{extst2}

    (a) If $\Ext_{\mathcal{OLA}}^i(L(\lambda),W(\mu))\neq 0$ then $\mu\leq_{fin}\lambda$;

    (b)  $\dim\Ext_{\mathcal{OLA}}^i(L(\lambda),W(\mu))<\infty $  { for all $i\geq 0$};

    (c)  $\Ext_{\mathcal{OLA}}^i(L(\lambda),W(\mu))=0$ for $i>|\mu_{{fin}}^+|$.  
    \end{corollary}

\begin{proposition}\label{BGG} (Analogue of BGG reciprocity) The multiplicity $(I(\mu):W(\lambda))$ equals  $m(\lambda,\mu)$. 
\end{proposition}
\begin{proof} Follows from the identity
  $$[M_n(\lambda):L(\mu)]=\dim\Hom_{\fg}(M_n(\lambda),I(\mu))=(I(\mu):W(\lambda)),$$
  where the second equality is a consequence of Lemma \ref{lem:coVerma}, (c).
  \end{proof}

\subsection{Jordan--H\"older multiplicities for standard objects}
Now we calculate the multiplicities of $[W(\lambda):L(\nu)]$.
We start {by} computing $\Phi_n(W(\lambda))$.

Recall the Lie subalgebra $\fs_n\subset \fk_n$. Consider the $\fs_n$-module
$$R(n,p):=\begin{cases}\bigoplus_{\mu\in\mathcal{P},\:|\mu|=p}\mathbb S_\mu(V^L_n)^*\boxtimes \mathbb S_\mu(V^R_n)\,\,\text{for}\,\,\fg=\mathfrak{sl}(\infty)\\
  \bigoplus_{\mu\in \mathcal{P},\:|\mu|=p}\mathbb S_{2\mu}(V_n^*)\,\,\text{for}\,\,\fg=\mathfrak{o}(\infty)\\
  \bigoplus_{\mu \in \mathcal{P},\:|\mu|=p}\mathbb S_{(2\mu)'}(V^*_n)\,\,\text{for}\,\,\fg=\mathfrak{sp}(\infty)
\end{cases}$$
where
$V_n$, $V^L_n$ and $V_n^R$ are introduced in the preamble to Section 4, $\mathcal{P}$ stands for the set of all partitions, and the superscript $'$ indicates conjugating a partition (transposing the corresponding Young diagram). Fix a decomposition of $\fh_n$-modules $\bar\fb\cap\fk_n=(\bar\fb\cap\fs_{n})\oplus\mathfrak{z}_n$   and set $\mathfrak z_n R(n,p)=0$ {in order} to define a $\bar\fb\cap\fk_n$-module structure  on $R(n,p)$.

\begin{lemma}\label{lem:phi} For sufficiently large $n$ there is an isomorphism of $\fk_n$-modules
  $$\Phi_n(W(\lambda))\simeq \bigoplus_{p\geq 0}\operatorname{Coind}^{\fk_n}_{\bar\fb\cap\fk_n}(R(n,p)\otimes \mathbb C_\lambda).$$
  \end{lemma}
  \begin{proof} First, we have isomorphisms of $\fk_n\oplus \fg_{n,n}$-modules
    $$\operatorname{Coind}^\fg_{\bar\fb}\mathbb C_\lambda\simeq
    \operatorname{Coind}^{\fk_n\oplus\fg_{n,n}}_{\bar\fb\cap(\fk_n\oplus\fg_{n,n})}\operatorname{Hom}_{\mathbb C}({\bf S}(\fr_n),\mathbb C_\lambda)\simeq$$
    $$\operatorname{Coind}^{\fk_n}_{\bar{\fb}\cap\fk_n}\operatorname{Hom}_{\mathbb C}({\bf S}(\fr_n),\operatorname{Coind}^{\fg_{n,n}}_{\bar{\fb}\cap{\fg}_{n,n}}\mathbb C_\lambda),$$
 where the structure of $\bar\fb\cap(\fk_n\oplus\fg_{n,n})$-module on  $\operatorname{Hom}_{\mathbb C}({\bf S}(\fr_n),\mathbb C_\lambda)$   comes from the isomorphism $\fr_n\simeq\fg/(\fk_n+\fg_{n,n}+\bar\fb)$.
    
    Recall that the result of application of $\Phi_n$ depends only on the restriction to $\fg_{n,n}$. Therefore
    $$\Phi_n(\operatorname{Coind}^{\fk_n}_{\bar{\fb}\cap\fk_n}\operatorname{Hom}_{\mathbb C}({\bf S}(\fr_n),\operatorname{Coind}^{\fg_{n,n}}_{\bar{\fb}\cap{\fg}_{n,n}}\mathbb C_\lambda))
    \simeq \operatorname{Coind}^{\fk_n}_{\bar{\fb}\cap\fk_n}\Phi_n(\operatorname{Hom}_{\mathbb C}({\bf S}(\fr_n),\operatorname{Coind}^{\fg_{n,n}}_{\bar{\fb}\cap{\fg}_{n,n}}\mathbb C_\lambda)).$$
    Furthermore, we have

    $$\Phi_n(\operatorname{Hom}_{\mathbb C}({\bf S}(\fr_n),\operatorname{Coind}^{\fg_{n,n}}_{\bar{\fb}\cap{\fg}_{n,n}}\mathbb C_\lambda))\simeq$$
    $$\operatorname{Hom}_{\fg_{n,n}}({\bf S}(\fr_n),\operatorname{Coind}^{\fg_{n,n}}_{\bar{\fb}\cap{\fg}_{n,n}}\mathbb C_\lambda)\simeq
    \operatorname{Hom}_{\bar{\fb}\cap{\fg}_{n,n}}({\bf S}(\fr_n),\mathbb C_\lambda) .$$
    Since ${\bf S}(\fr_n)$ is a direct sum of objects from $\mathbb T_{\fg_{n,n}}$ and $\mathbb C_{\lambda}$ is a trivial $\fg_{n,n}$-module, we have
    $$\operatorname{Hom}_{\bar{\fb}\cap{\fg}_{n,n}}({\bf S}(\fr_n),\mathbb C_\lambda) \simeq \operatorname{Hom}_{{\fg}_{n,n}}({\bf S}(\fr_n),\mathbb C_\lambda ).$$

  Next, we observe the following $\fs_n\oplus\fg_{n,n}$-module isomorphism
  $$\fr_n\simeq \begin{cases}V_n^L\boxtimes (\bar {V_n})_*\oplus (V_n^R)^*\boxtimes \bar {V_n}\,\,\text{for}\,\,\fg=\mathfrak{sl}(\infty)\\
V_n\boxtimes \bar {V_n}\,\,\text{for}\,\,\fg=\mathfrak{o}(\infty),\mathfrak{sp}(\infty).
\end{cases}$$

To finish the proof we have to show that  
\begin{equation}\label{5.2}
\operatorname{Hom}_{\fg_{n,n}}({\bf S}(\fr_n),\mathbb C_\lambda)\simeq  \bigoplus_{p\geq 0} R(n,p)\otimes\mathbb{C}_{\lambda}.\end{equation}

If $\fg=\mathfrak{sl}(\infty)$ then
$${\bf S}(V_n^L\boxtimes (\bar {V}_n)_*\oplus (V_n^R)^*\boxtimes \bar {V}_n)=\bigoplus_{\mu,\nu\in\mathcal{P}}(\mathbb{S}_\mu (V_n^L)\boxtimes\mathbb{S}_\mu((\bar {V_n})_*))\otimes(\mathbb{S}_\nu(V_n^R)^*\boxtimes \mathbb{S}_\nu(\bar {V_n})).$$
Since $$\operatorname{Hom}_{\mathfrak g_{n,n}}(\mathbb{S}_\mu(\bar {(V_n})_*)\otimes  \mathbb{S}_\nu(\bar {V_n}),\mathbb C)=\begin{cases}\mathbb C\,\text{for}\,\mu=\nu\\ 0\,\text{for}\,\mu\neq\nu,\end{cases}$$
we obtain
$$\operatorname{Hom}_{\fg_{n,n}}({\bf S}(\fr_n),\mathbb C_\lambda)=\bigoplus_{\mu \in \mathcal{P}}\mathbb{S}_\mu (V_n^L)^*\otimes \mathbb{S}_\mu(V_n^R)\otimes \mathbb C_\lambda.$$

If $\fg=\mathfrak{o}(\infty)$, $\mathfrak{sp}(\infty)$  then
$${\bf S}(V_n\boxtimes \bar {V}_n)=\bigoplus_{\nu \in \mathcal{P}}\mathbb{S}_\nu(V_n)\boxtimes \mathbb{S}_\nu(\bar {V_n}).$$
Since
$$\operatorname{Hom}_{\mathfrak g_{n,n}}(\mathbb{S}_\nu(\bar {V}_n),\mathbb C)=\begin{cases}\begin{cases}\mathbb C\,\text{if}\,\nu=2\mu\\ 0\,\text{otherwise}\end{cases}\text{ for }\fg=\mathfrak{o}(\infty)\\\begin{cases}\mathbb C\,\text{if}\,\nu=(2\mu)'\\ 0\,\text{otherwise}\end{cases}\text{ for }\fg=\mathfrak{sp}(\infty),\end{cases}$$
we obtain
$$\operatorname{Hom}_{\fg_{n,n}}({\bf S}(\fr_n),\mathbb C_\lambda)=\begin{cases}\bigoplus_{\mu \in \mathcal{P}}\mathbb S_{2\mu}(V_n^*)\otimes\mathbb C_\lambda \text{ for }\fg=\mathfrak{o}(\infty), \\\bigoplus_{\mu \in \mathcal{P}}\mathbb S_{(2\mu)'}(V_n^*)\otimes\mathbb C_\lambda \text{ for }\fg=\mathfrak{sp}(\infty).\end{cases}$$

In both cases we have now established (\ref{5.2}), and the statement follows.
\end{proof}

 Let $W_{\fs}(\lambda)$ be a standard object in the category $\mathcal{OLA}(\fs)$: its definition is the obvious analogue of the definition of $W(\lambda)$. Next, we define the $\fs$-modules $R(\infty,k)$ by setting

 $$R(\infty,k):=\begin{cases}\bigoplus_{\mu \in \mathcal{P},\:|\mu|=k}\mathbb S_\mu(V^L_*)\boxtimes \mathbb S_\mu(V^R)\,\,\text{for}\,\,\fg=\mathfrak{sl}(\infty)\\
  \bigoplus_{\mu \in \mathcal{P},\:|\mu|=k}\mathbb S_{2\mu}(V_*)\,\,\text{for}\,\,\fg=\mathfrak{o}(\infty)\\
  \bigoplus_{\mu \in \mathcal{P},\:|\mu|=k}\mathbb S_{(2\mu)'}(V_*)\,\,\text{for}\,\,\fg=\mathfrak{sp}(\infty),
\end{cases}$$
where $$V^L_*:= \lim_{\longrightarrow} (V^L_n)^*, \;V^R:=\lim_{\longrightarrow}V^R_n .$$

We are now ready to describe the  canonical filtration (\ref{canonical}) of the standard objects $W(\lambda)$.
Let  $\Gamma_{\fh_n}$ denote  the endofunctor of $\fh_n$-semisimple vectors on the category $\fk_n$-mod.
Define the ${\fk_n}$-module
  $$S(n,p,\lambda):=\Gamma_{\fh_n}(\operatorname{Coind}^{\fk_n}_{\bar\fb\cap\fk_n}(R(n,p)\otimes \mathbb C_\lambda)).$$

\begin{proposition}\label{cor:stand} 
 There are isomorphisms of $\fg$-modules
  $$W(\lambda)\simeq\lim_{\longrightarrow}{(}\bigoplus_{p\geq 0}S(n,p,\lambda){)},$$
  $$D_k(W(\lambda))\simeq\lim_{\longrightarrow}{(}\bigoplus_{0\leq p\leq k-1}S(n,p,\lambda){)},$$
  and
  $$\left(D_{k+1}(W(\lambda))/D_{k}(W(\lambda))\right)^+\simeq R(\infty,k)\otimes W_{\fs}(\lambda).$$
  \end{proposition}
  \begin{proof} The first isomorphism follows from Lemma \ref{lem:phi} and the identity
    $\Phi_n\circ\Gamma_{\fh}=\Gamma_{\fh_n}\circ\Phi_n$. 

    To verify the existence of the second isomorphism, we first observe that
    $$\operatorname{Hom}_{\fk_n}(S(n,p,\lambda), S(n+1,q,\lambda))= 0\,\,\text{if}\,\,p>q,$$ as follows from a direct comparison of supports.
    Hence $W(\lambda)$ has an ascending exhaustive filtration $0=F_0\subset F_1\subset F_2\subset\dots$ with
    $$F_p/F_{p-1}\simeq \lim_{\longrightarrow}S(n,p-1,\lambda).$$ 

We claim that
    $F_p=D_p(W(\lambda))$. To prove this it suffices to check that $\displaystyle\lim_{\longrightarrow}S(n,p,\lambda)$ is an object of $\mathcal{OLA}^{d+p}$, where $d=d(\lambda)$.
    Indeed, $S(n,p,\lambda)$ has a filtration with quotients isomorphic to $W_{\fk_n}(\lambda+\gamma)$ for all weights $\gamma$ of $R(n,p)$. Note that $d(\gamma)=p$.
    If $$[\lim_{\longrightarrow}S(n,p,\lambda):L(\mu)]\neq 0,$$ then there exists a weight $\gamma$ of $R(n,p)$ such that
    $[W_{\fk_n}(\lambda+\gamma):L_{\fk_n}(\mu)]\neq 0$ for  all sufficiently large $n$. Since the character
    of $W_{\fk_n}(\lambda)$ coincides with  the character of $M_{\fk_n}(\lambda)$, by the same argument as in Lemma \ref{order1}
    we obtain that $\lambda+\gamma=\mu$ or $\lambda+\gamma-\mu$ is a sum of positive finite roots. Hence $d(\mu)=d(\lambda+\gamma)=d+p$.

    Finally, let's establish the third isomorphism. Define the functor $T:\mathcal {O}^d_{\fk_n}\to \mathcal{O}_{\fs_n}$ by setting
    $$T_d(N):=\oplus_{\nu\in\operatorname{supp} N, \;d(\nu)=d}N_{\nu}.$$
    Then $M^+=\displaystyle\lim_{\longrightarrow}{(}T_d\circ\Phi_n(M){)}$ for $M\in \mathcal{OLA}^d$. In particular,
    $$\left(D_{k+1}(W(\lambda))/D_{k}(W(\lambda))\right)^+=\lim_{\longrightarrow}{(}T_{d+k}(S(n,k,\lambda){)}\simeq\lim_{\longrightarrow}R(n,k)\otimes W_{\fs}(\lambda)=R(\infty,k)\otimes W_{\fs}(\lambda).$$
    \end{proof}

Before stating the main result of this subsection we need to introduce some further notation. Consider the $\fs$-module
$$R:=\bigoplus_{k\geq 0}R(\infty,k),$$
and denote by $\mathcal R$ (respectively, $\mathcal R_k$) the support of $R$  (respectively, $R(\infty,k)$).

If $\fg=\mathfrak{sl}(\infty)$, then all $\gamma\in\mathcal R_k$ are of the form $\gamma^L+\gamma^R$,
 where $\gamma^L=\sum_{i\in \mathbb Z_{>0}}a_i\varepsilon_i$  and
 $\gamma^R=\sum_{i\in \mathbb Z_{<0}}b_i\varepsilon_i$ for some $a_i\in\mathbb Z_{\leq 0}, b_i\in\mathbb Z_{\geq 0}$ such that $-\sum a_i=\sum b_i=k$.
For $\fg=\mathfrak{o}(\infty),\mathfrak{sp}(\infty)$,  every $\gamma\in\mathcal R_k$ can be written uniquely in the form
 $\gamma=\sum_{i>0} a_i\varepsilon_i$ with non-positive integers $a_i$ such that $\sum a_i=-2k$.

  Let $\mu$ be a partition. By $K(\mu,\gamma)$ we denote the multiplicity of a weight $\gamma$ in the
$\mathfrak{sl}(\infty)$-module $\mathbb{S}_{\mu}(V)$. If $\gamma$ is also a partition then $K(\mu,\gamma)$ are Kostka numbers by definition. In fact, $K(\mu,\gamma)$ are always Kostka numbers as  $K(\mu,\gamma)= K(\mu, w(\gamma))$ for $w\in \mathcal{W}$, and for any given $\gamma\in$ supp$\:\mathbb{S}_{\mu}(V)$ there is a suitable $w\in \mathcal{W}$ for which $w(\gamma)$ is a partition. By $\mathcal P_{ev}$ we denote the set of even partitions and by $\mathcal P'_{ev}$ the set of all partitions
whose conjugates are even partitions. 
 
\begin{proposition}\label{mult}

\begin{enumerate}[(a)]
\item If $\fg=\mathfrak{sl}(\infty)$, then
 $$[W(\lambda):L(\nu)]=\sum_{\mu\in\mathcal P,\gamma\in \mathcal R}  K(\mu,-\gamma^L)K(\mu,\gamma^R)m(\lambda+\gamma,\nu).$$
\item If $\fg=\mathfrak{o}(\infty)$, then
 $$[W(\lambda):L(\nu)]=\sum_{\mu\in\mathcal P_{ev},\gamma\in\mathcal R} K(\mu,-\gamma)m(\lambda+\gamma,\nu).$$
\item If $\fg=\mathfrak{sp}(\infty)$, then
 $$[W(\lambda):L(\nu)]=\sum_{\mu\in\mathcal P'_{ev},\gamma\in\mathcal R} K(\mu,-\gamma)m(\lambda+\gamma,\nu).$$
\end{enumerate}

\end{proposition}
\begin{proof} The proposition follows from Proposition \ref{cor:stand} and Lemma \ref{lem:multred}. Indeed, let $d(\nu)=d(\lambda)+k$. Then
  $$[W(\lambda):L(\nu)]=[\left(D_{k+1}(W(\lambda))/D_k(W(\lambda))\right)^+:L_\fs(\nu)]=[R(\infty,k)\otimes W_{\fs}(\lambda):L_\fs(\nu)].$$
  Since $R(\infty,k)\otimes W_{\fs}(\lambda)$  has a filtration with quotients isomorphic to $ W_{\fs}(\lambda+\gamma)$ where $\gamma$ runs  over $\mathcal R_k$, the multiplicity $(R(\infty,k)\otimes W_{\fs}(\lambda):W_\fs(\lambda+\gamma))$ equals the multiplicity $c_k(\gamma)$ of the weight $\gamma$ in $R(\infty,k)$. Therefore
  \begin{equation}\label{eq:mult}[R(\infty,k)\otimes W_{\fs}(\lambda):L_\fs(\nu)]=\sum_{\gamma} c_k(\gamma)[W_{\fs}(\lambda+\gamma):L(\nu)]=\sum_{\gamma}c_k(\gamma)m(\lambda+\gamma,\nu).
    \end{equation}
  The statement now follows from an explicit calculation of $c_k(\gamma)$:
  $$c_k(\gamma)=\begin{cases} \sum_{\mu\in \mathcal{P},\:|\mu|=k}K(\mu,-\gamma^L)K(\mu,\gamma^R)\,\,\text{for}\,\,\fg=\mathfrak{sl}(\infty),\\
    \sum_{\mu\in\mathcal{P}_{ev},\: |\mu|=2k}K(\mu,-\gamma)\,\,\text{for}\,\,\fg=\mathfrak{o}(\infty),\\
    \sum_{\mu\in\mathcal{P}'_{ev},\: |\mu|=2k}K(\mu,-\gamma)\,\,\text{for}\,\,\fg=\mathfrak{sp}(\infty).
      \end{cases}$$
  \end{proof}

\subsection{Highest weight category}
We are now ready to define a new partial order $\leq_{inf}$  on the set of eligible weights. This is the partial order needed for the structure of highest weight category on $\mathcal{OLA}$.
We write $\mu\lhd_{inf}\nu$ if one of the following holds:
\begin{enumerate}[(i)]
\item $\mu=\nu+\gamma$ for some $\gamma\in\mathcal R$,
  \item $\mu\leq_{fin}\nu$.
  \end{enumerate}
 
  By definition, the partial order $\leq_{inf}$ is the reflexive and transitive closure of the relation $\lhd_{inf}$.
 \begin{remark}
   Note that $\mu\leq_{inf}\nu$ whenever $\mu\leq_{fin}\nu$. Furthermore, $\mu\leq_{inf}\nu$ implies $d(\mu)\leq d(\nu)$.
   Finally, it is a consequence of the formula (\ref{eq:mult}) that
  \begin{equation}\label{eq:mult2}
    [W(\lambda):L(\nu)]\neq 0\,\,\,\Rightarrow\,\,\,\nu\leq_{inf}\lambda. 
  \end{equation}
  \end{remark}
   The condition (\ref{eq:mult2}) justifies introducing the partial order $\leq_{inf}$ as the inequality $\nu\leq_{fin}\lambda$ does not necessarily hold when $[W(\lambda):L(\nu]\neq 0$.  
  \begin{lemma}\label{locfin} The order $\leq_{inf}$ is interval-finite. 
  \end{lemma}
  \begin{proof} Let $\fg=\mathfrak{o}(\infty)$ or $\mathfrak{sp}(\infty)$. Then we can take
    $\rho=\sum_{i\geq 1}-i\varepsilon_i$. For an eligible weight $\lambda$, set $\tilde\lambda=\lambda+\rho$ and write
    $\tilde\lambda=\sum_{i\geq 1}\tilde\lambda_i\varepsilon_i$. Let $i\in\mathbb Z_{>0}$ and $m\in\mathbb Z$ be such that
    \begin{equation}\label{eq:cond} \operatorname{Re}\tilde\lambda_j\geq m\,\,\text{ for all}\,\, j\leq i.
      \end{equation}
      We claim that if $\kappa\leq_{inf}\mu$ and (\ref{eq:cond}) holds for $\kappa$ that it also holds for $\mu$. Indeed, it suffices to check this in
      two situations:
     \begin{itemize}
     \item $\tilde\mu=s_\alpha(\tilde\kappa)$ for some reflection $s_{\alpha}\in\mathcal W$ such that $\tilde\mu-\tilde\kappa\in\langle\alpha\rangle_{\mathbb Z_{>0}}$.
       \item $\tilde\mu=\tilde\kappa-\gamma$ for some $\gamma\in\mathcal R$.
       \end{itemize}
       In both cases the checking is straightforward and we leave it to the reader.

       Now we note that for any eligible $\lambda$ and $\mu$ there exists $n\in \mathbb Z_{>0}$ such that condition(\ref{eq:cond}) holds for  both $\lambda$ and $\mu$ whenever $i> n$
        and  $m=-i$.
       Then, if $\lambda\leq_{inf}\kappa\leq_{inf}\mu$ we have $\tilde\lambda_i=\tilde\kappa_i=\tilde\mu_i=-i$ for any $i> n$. 
Therefore, in order to check that for fixed $\lambda$ and $\mu$ there are at most finitely many $\kappa$ satisfying  $\lambda\leq_{inf}\kappa\leq_{inf}\mu$, it suffices to establish that there are at most finitely many possibilities for the restriction 
       $\kappa|_{\fh_n}$. But this follows from the well-known interval-finiteness of the standard weight order for the finite-dimensional reductive Lie algebra $\fk_n$.
       
In  the case of $\mathfrak{sl}(\infty)$ we apply the same argument to the weights  $\lambda^L$ and $\lambda^R$ separately. 
    \end{proof}

Finally, the implication (\ref{eq:mult2}) together  with Lemma \ref{locfin} yields the following.
    \begin{corollary} The category $\mathcal{OLA}$ is a highest weight category {according to Definition 3.1 in \cite{CPS},} with standard objects $W(\lambda)$ and partial order $\leq_{inf}$.
      \end{corollary}

\subsection{Blocks of $\mathcal{OLA}$}
Recall that $\langle\tilde I \rangle_{\mathbb C}$ is the set of eligible weights. Let $Q=\langle\Delta \rangle_{\mathbb Z}$ denote the root lattice.  For $\kappa\in\langle\tilde I \rangle_{\mathbb C}/Q$ we define
$\mathcal {OLA}_\kappa$ as the full subcategory of $\mathcal {OLA}$ consisting of modules  $M$ with   $\operatorname{supp}M\subset\kappa$.
Then obviously 
$$\mathcal{OLA}=\Pi_{\kappa\in \langle\tilde I \rangle_{\mathbb C}/Q}\mathcal {OLA}_\kappa.$$

The following theorem claims that blocks of $\mathcal{OLA}$ are "maximal possible" as two simple objects of $\mathcal{OLA}$ are in different blocks if and only if their supports are not linked by elements of the root lattice. This result is a generalization of the description of blocks of the category $\mathbb{T}_{\fg}$  \cite{DPS}, and  is in sharp contrast with the description of blocks in the classical BGG category $\mathcal{O}$.
\begin{theorem}\label{thm:blocks} The subcategory $\mathcal {OLA}_\kappa$ is indecomposable for any $\kappa \in \langle \tilde I\rangle_{\mathbb{C}}/Q$.
\end{theorem}
\begin{proof} We start by noticing that $\langle\mathcal R_1\rangle_{\mathbb Z}=Q$. Hence it suffices to prove that for any
  $\lambda \in \langle \tilde I\rangle_{\mathbb{C}}$ and any $\gamma\in \mathcal R_1$, the simple modules $L(\lambda)$ and $L(\lambda+\gamma)$ belong to the same block.
  This follows immediately from (\ref{eq:mult}) with $k=1$ since $[W(\lambda):L(\lambda)]=[W(\lambda):L(\lambda+\gamma)]=1$ and $W(\lambda)$ is indecomposable.
  \end{proof}

A block $\mathcal{OLA_\kappa}$ is  {\it integral} if it contains $L(\lambda)$ for some $\lambda\in\langle\tilde I\rangle_{\mathbb Z}$ (equivalently, such that $\frac{2(\lambda,\alpha)}{(\alpha,\alpha)}\in \mathbb Z$ for any  $\alpha\in \Delta$).
\begin{corollary}
The integral blocks of $\mathcal{OLA}$ are parametrized by $\mathbb Z$ for $\fg=\mathfrak{sl}(\infty)$, and by $\mathbb Z / 2 \mathbb Z$ for $\fg=\mathfrak{o}(\infty)$, $\mathfrak{sp}(\infty)$.
\end{corollary}

\section{Annihilators in $U(\fg)$ of objects of $\mathcal{OLA}$}

In this short final section we discuss the annihilators in $U(\fg)$ of the objects of $\mathcal{OLA}$. We restrict ourselves to the case $\fg=\mathfrak{sl}(\infty)$. Recall that, according to Theorem 7.1 in \cite{PP2}, the primitive ideals of $U(\mathfrak{sl}(\infty))$ are parametrized by quadruples $(x, y, Y_l, Y_r)$  where $x, y$ run over $\mathbb{Z}_{\geq0}$ and $Y_l$, $Y_r$ are arbitrary partitions. The parameter $x$ comes from the characteristic pro-variety of the ideal \cite{PP1} and is called \emph{rank}, while the parameter $y$ is the \emph{Grassmann number}. In the paper \cite{PP3} an algorithm for computing the annihilator of an arbitrary  simple highest weight $\mathfrak{sl}(\infty)$-module is presented. A significant difference with the case of a finite-dimensional Lie algebra is that the annihilators of most simple highest weight $\mathfrak{sl}(\infty)$-modules equal zero in $U(\mathfrak{sl}(\infty))$.

Furthermore, it is a direct observation based  on Theorem 7.1 in \cite{PP2} that, for  a simple object $L(\lambda)$ of $\mathcal{OLA}$ the annihilator Ann$_{U(\fg)}L(\lambda)$ is nonzero and has the form $I(x,0,Y_l,Y_r)$ for some $x$, $Y_l$ and $Y_r$. In particular,  the annihilators of simple objects of $\mathcal{OLA}$ have Grassmann number equal to zero.
\begin{corollary}\label{last}
Let $\fg=\mathfrak{sl}(\infty)$. If $M$ is a finitely generated object of $\mathcal{OLA}$, then $\operatorname{Ann}_{U(\fg)}M\neq0$.
\end{corollary} 
\begin{proof} By Proposition \ref{fl}, any finitely generated module in $\mathcal{OLA}$ has finite length. By the above observation, the annihilator in $U(\fg)$ of any simple module in $\mathcal{OLA}$ is nonzero. Finally, it is an exercise to check, using Theorem 5.3 in \cite{PP2},  that the intersection of finitely many primitive ideals of $U(\fg)$ is nonzero.
\end{proof}
We conjecture that Corollary \ref{last} holds for $\fg=\mathfrak{o}(\infty), \mathfrak{sp}(\infty)$, but in these cases the algorithm for computing the primitive ideal of a simple highest weight module is still in progress.

 If $L(\lambda)\in \mathcal{OLA}$ is integrable, then Ann$_{U(\fg)}L(\lambda)=I(0,0,\lambda^1,\lambda^2)$ where $\lambda^1$ and $\lambda^2$ are the two partitions comprising $\lambda$, see Subsection \ref{Simple modules}. Moreover, a simple module $L(\lambda)\in\mathcal{OLA}$ is not integrable precisely when Ann$_{U(\fg)}L(\lambda)=I(x,0,Y_l,Y_r)$ for $x\neq 0$. This follows from a result of A. Sava \cite{Sava} but also from a direct application of the algorithm of \cite{PP2}. 
In fact, all primitive ideals of the form $ I(x,0,Y_l, Y_r)$ are annihilators of simple objects of $\mathcal{OLA}$.       Indeed,  the reader will verify immediately using Theorem 7.1 in \cite{PP3} that, given $x\in \mathbb Z_{\geq 0}$ and partitions $Y_l=(y^l_1, y^l_2, \dots,y^l_k)$, $Y_r=(y^r_1, y^r_2, \dots,y^r_s)$, we have
$$\operatorname{Ann}_{U(\fg)}L(\lambda) = I(x,0,Y_l, Y_r)$$
for $\lambda:= \lambda^L+\lambda^R$, $\lambda^L=\sum^x_{i=1} a_i \varepsilon_i + \sum^k_{i=1}y^l_i\varepsilon_{x+i}$, $\lambda^R= -\sum^s_{i=1}y^r_{s+1-i} \varepsilon_{-i}$, where $a_1,\dots a_x$ are complex numbers satisfying the conditions $a_i\not \in \mathbb{Z}$, $a_i-a_j\not\in \mathbb Z$ for all $i,j$.


\begin{thebibliography}{20}


\bibitem[BGG]{BGG} J. Bernstein, I.  Gelfand, S. Gelfand, {A category of $\mathfrak{g}$-modules}, Funktional Anal. i Prilozhen {\bf5} (1971), No. 2 , 1--8; English translation, Functional Anal. and Appl. {\bf10} (1976), 87--92.
\bibitem [CP]{CP} K. Coulembier, I. Penkov,  {On an infinite limit of BGG Categories O}, arXiv:1802.06343.
\bibitem[CPS]{CPS} E. Cline, B. Parshall, L. Scott, {Finite-dimensional algebras and highest weight categories}, J. Reine Angew. Math. \textbf{391} (1988), 85-99.
\bibitem[DP]{DP} I. Dimitrov, I. Penkov, {Weight modules of direct limit Lie algebras}, IMRN 1999, no. 5, 223-249.

\bibitem [DPS] {DPS} E. Dan-Cohen, I. Penkov, V. Serganova, {A Koszul category of representations of finitary Lie algebras}, Advances of Mathematics {\bf 289} (2016), 250--278.


\bibitem [Fe] {Fe}  S. Fernando, {Lie algebra modules with finite dimensional weight spaces I}, Transactions of AMS \textbf{322} (1990), 757--781.
\bibitem [K] {K} V. Kac, {Constructing groups associated to  infinite-dimensional Lie algebras}. In: Infinite-dimensional groups with applications, MSRI Publications, vol. 4, 1985, 167-216.    
\bibitem[Mac]{Mac} G. Mackey, {On infinite-dimensional linear spaces},  Transactions of AMS {\bf57} (1945), 155--207.
	\bibitem[N]{Nam}
	T.~Nampaisarn, On categories O for root-reductive Lie algebras,  arXiv:1711.11234.
\bibitem[PP1]{PP1} I. Penkov, A. Petukhov, {On ideals in the enveloping algebra of a locally simple Lie algebra},  Int. Math. Res. Notices  {\bf 2015}, 5196-5228.
\bibitem[PP2]{PP2} I. Penkov, A. Petukhov, {Primitive ideals of $U(\mathfrak{sl}(\infty))$}, Bulletin LMS {50 (2018), 443-448}.
\bibitem[PP3]{PP3} I. Penkov, A. Petukhov, {Primitive ideals of $U(\mathfrak{sl}(\infty))$ and the Robinson-Schensted algorithm at infinity}, {to appear in Representation of Lie Algebraic Systems and Nilpotent orbits, Progress in Mathematics, Birkhauser}, arXiv:1801.06692.
\bibitem[PS]{PS} I. Penkov, V. Serganova, {Tensor representations of Mackey Lie algebras and their dense subalgebras}. In: Developments and Retrospectives in Lie Theory: Algebraic Methods, Developments in Mathematics, vol. 38, Springer Verlag, 2014, 291-330.
\bibitem [PStyr] {PStyr} I. Penkov, K. Styrkas, {Tensor representations of infinite-dimensional root-reductive Lie algebras}. In: Developments and Trends in Infinite-Dimensional Lie Theory, Progress in Mathematics, vol. 288, Birkh\"auser, 2011, 127--150.
\bibitem [PSZ] {PSZ} I. Penkov, V. Serganova, G. Zuckerman, {On the existence of $(\mathfrak{g},\mathfrak{k})$-modules of finite type}, Duke Math. J. \textbf{125} (2004), 329--349.
\bibitem [S]{Sava} A. Sava, {Annihilators of simple tensor modules}, master{'}s thesis, Jacobs University Bremen, 2012, arXiv: 1201.3829.
\end{thebibliography}
\end{document}